\documentclass[a4paper,12pt]{amsart}
\usepackage{a4wide}
\usepackage{amsfonts}
\usepackage{amssymb}
\usepackage{amsmath}
\usepackage{amsthm}
\usepackage{bbm}
\usepackage{verbatim}
\usepackage{url}
\usepackage[latin1]{inputenc} 
\usepackage[T1]{fontenc}

\usepackage{amsfonts}
\usepackage{amsmath,amssymb,amsthm}

\usepackage[colorlinks, citecolor=blue,hypertexnames=false]{hyperref}

\usepackage{mathrsfs}
\usepackage{verbatim}

\usepackage{graphicx}
\usepackage{ifpdf}

\DeclareMathOperator{\ch}{Ch}

\newtheorem{thm}{Theorem}[section]
\newtheorem{prop}[thm]{Proposition}

\newtheorem{lemma}[thm]{Lemma}

\theoremstyle{definition} 

\newtheorem{defn}[thm]{Definition}

\newtheorem{remark}[thm]{Remark}

\newcommand{\omegareverse}{\frac{\displaystyle 1}{\displaystyle \mu(\Omega)}}

\newcommand{\m}{\medskip}

\newcommand{\al}{\alpha}

\newcommand{\om}{\omega}

\newcommand{\wX}{\widetilde{X}}

\newcommand{\R}{\mathbb{R}}
\newcommand{\Z}{\mathbb{Z}}
\newcommand{\N}{\mathbb{N}}

\newcommand{\D}{\mathscr{D}}

\newcommand{\XX}{X}

\newcommand{\bmo}{{\rm BMO}}

\newcommand{\intav}{-\!\!\!\!\!\!\int}

\newcommand{\GG}{\mathop G \limits^{    \circ}}

\newcommand{\eps}[0]{\varepsilon}

\newcommand{\E}{\mathbb{E}}
\newcommand{\Db}{\mathbb{D}}

\newcommand{\Norm}[2]{\|#1\|_{#2}}
\newcommand{\loc}[0]{\operatorname{loc}}
\newcommand{\lspan}[0]{\operatorname{span}}
\providecommand{\abs}[1]{\lvert#1\rvert}

\DeclareMathOperator{\supp}{supp}

\DeclareMathOperator*{\esssup}{ess\ sup}

\def\XXint#1#2#3{{\setbox0=\hbox{$#1{#2#3}{\int}$}
     \vcenter{\hbox{$#2#3$}}\kern-.5\wd0}}

\allowdisplaybreaks

\begin{document}

\title[Haar bases and dyadic structure theorems]{Haar
    bases on quasi-metric measure spaces, and \\
    dyadic structure theorems for function spaces on \\
    product spaces of homogeneous type}

\author{Anna Kairema}
\address{Department of Mathematics and Statistics\\
         P.O.B.~68 (Gustaf H\"allstr\"omin katu 2)\\
         FI-00014 University of Helsinki\\
         Finland
         }
\email{anna.kairema@helsinki.fi}

\author{Ji Li}
\address{Department of Mathematics\\
         Macquarie University\\
         NSW 2019\\
         Australia
         }
\email{ji.li@mq.edu.au}

\author{M.~Cristina Pereyra}
\address{Department of Mathematics and Statistics\\
         University of New Mexico\\
         Albuquerque, NM 87131, USA
         }
\email{crisp@math.unm.edu}

\author{Lesley A.~Ward}
\address{School of Information Technology and Mathematical Sciences\\
         University of South Australia\\
         Mawson Lakes  SA  5095\\
         Australia
         }
\email{lesley.ward@unisa.edu.au}

\thanks{The second and fourth authors were supported by the
Australian Research Council, Grant No.~ARC-DP120100399. Part of
the results in this paper was done during the second author's
working in Sun Yat-sen University supported by the NNSF of
China, Grant No.~11001275, by a China Postdoctoral Science
Foundation funded project, Grant No.~201104383, and by the
Fundamental Research Funds for the Central Universities, Grant
No.~11lgpy56.}



\date{\today}



\begin{abstract}
    We give an explicit construction of Haar
    functions associated to a system of dyadic cubes in a
    geometrically doubling quasi-metric space equipped with a
    positive Borel measure, and show that these Haar
    functions form a basis for~$L^p$. Next we focus on
    spaces~$X$ of homogeneous type in the sense of Coifman and
    Weiss, where we use these Haar functions to define a
    discrete square function, and hence to define dyadic
    versions of the function spaces $H^1(X)$ and~$\bmo(X)$. In
    the setting of product spaces $\widetilde{X} = X_1 \times
    \cdots \times X_n$ of homogeneous type, we show that the
    space $\bmo(\widetilde{X})$ of functions of bounded mean
    oscillation on~$\widetilde{X}$ can be written as the
    intersection of finitely many dyadic $\bmo$ spaces
    on~$\widetilde{X}$, and similarly for $A_p(\widetilde{X})$,
    reverse-H\"older weights on~$\widetilde{X}$, and doubling
    weights on~$\widetilde{X}$. We also establish that the
    Hardy space~$H^1(\widetilde{X})$ is a sum of finitely many
    dyadic Hardy spaces on~$\widetilde{X}$, and that the strong
    maximal function on~$\widetilde{X}$ is pointwise comparable
    to the sum of finitely many dyadic strong maximal
    functions. These dyadic structure theorems generalize, to
    product spaces of homogeneous type, the earlier Euclidean
    analogues for $\bmo$ and $H^1$ due to Mei and to Li, Pipher
    and Ward.
\end{abstract}

\maketitle

\tableofcontents



\section{Introduction}
\label{sec:introduction}
\setcounter{equation}{0}





This paper has three main components: (1)~the explicit
construction of a Haar basis associated to a system of dyadic
cubes on a geometrically doubling quasi-metric space equipped
with a positive Borel measure, (2)~definitions of dyadic
product function spaces, by means of this Haar basis, on
product spaces of homogeneous type in the sense of Coifman and
Weiss, and (3)~dyadic structure theorems relating the
continuous and dyadic versions of these function spaces on
product spaces of homogeneous type. We describe these
components in more detail below.

The function spaces we deal with in this paper are the Hardy
space~$H^1$, the space $\bmo$ of functions of bounded mean
oscillation, and the classes of Muckenhoupt $A_p$~weights,
reverse-H\"older weights, and doubling weights. Much of the
theory of these function spaces can be found in~\cite{JoNi},
\cite{FS}, \cite{GCRF} and~\cite{S}. Here we are concerned with
generalising the dyadic structure theorems established for
Euclidean underlying spaces in \cite{LPW} and~\cite{Mei}.

A central theme in modern harmonic analysis has been the drive
to extend the Calder\'on--Zygmund theory from the Euclidean
setting, namely where the underlying space is $\R^n$ with the
Euclidean metric and Lebesgue measure, to more general
settings. To this end Coifman and Weiss formulated the concept
of spaces of homogeneous type, in~\cite{CW1}. Specifically, by
a \emph{space~$(X,\rho,\mu)$ of homogeneous type} in the sense
of Coifman and Weiss, we mean a set $X$ equipped with a
quasi-metric~$\rho$ and a Borel measure~$\mu$ that is doubling.
There is a large literature devoted to spaces of homogeneous
type, in both their one-parameter and more recently
multiparameter forms. See for
example~\cite{CW1,CW2,MS,DJS,H1,HS,H2,Aimar2007,DH,HLL,CLW}
and~\cite{AuHyt}. Some non-Euclidean examples of spaces of
homogeneous type are given by the Carnot--Carath\'eodory spaces
whose theory is developed by Nagel, Stein and others
in~\cite{NS1}, \cite{NS2} and related papers; there the
quasi-metric is defined in terms of vector fields satisfying
the H\"ormander condition on an underlying manifold.

Meyer noted in his preface to~\cite{DH} that \emph{``One is
amazed by the dramatic changes that occurred in analysis during
the twentieth century. In the 1930s complex methods and Fourier
series played a seminal role. After many improvements, mostly
achieved by the Calder\'on--Zygmund school, the action takes
place today on spaces of homogeneous type. No group structure
is available, the Fourier transform is missing, but a version
of harmonic analysis is still present. Indeed the geometry is
conducting the analysis.''}

When we go beyond the Euclidean world and attempt to prove our
dyadic structure theorems in the setting of spaces of
homogeneous type, we immediately encounter the following two
obstacles. First, the Euclidean proofs rely on the so-called
one-third trick, which says in effect that each ball is
contained in some cube whose measure is (uniformly) comparable
to that of the ball, and which belongs either to the usual
dyadic lattice or to a fixed translate of the dyadic lattice.
However, in a space of homogeneous type there is no notion of
translation. Second, in the Euclidean proofs one uses a
definition of dyadic function spaces in terms of the Haar
coefficients. However, the theory of Haar bases on spaces of
homogeneous type is not fully developed in the literature. In
the present paper, we overcome the first obstacle by means of
the adjacent systems of dyadic cubes constructed by the first
author and Hyt\"onen in~\cite{HK}, and the second obstacle by
our full development of an explicit construction of Haar bases
for $L^p$ in a setting somewhat more general than that of
spaces of homogeneous type.

We now describe in more detail the three main components of
this paper, listed at the start of the introduction.

(1)~We include in this paper a detailed construction of a Haar
basis associated to a system of dyadic cubes on a geometrically
doubling quasi-metric space equipped with a positive Borel
measure. The geometric-doubling condition says that each ball
can be covered by a uniformly bounded number of balls of half
the radius of the original ball. This setting is somewhat more
general than that of spaces of homogeneous type in the sense of
Coifman and Weiss. Although such bases have been discussed in
the existing literature (see below), to our knowledge a full
construction has not appeared before. We believe that such
bases will also be useful for other purposes, beyond the
definitions of dyadic function spaces for which they are used
in this paper.

Haar-type bases for $L^2(X,\mu)$ have been constructed in
general metric spaces, and the construction is well known to
experts. Haar-type wavelets associated to nested partitions in
abstract measure spaces are provided by Girardi and Sweldens
\cite{GirardiSweldens1997}. Such Haar functions are also used
in~\cite{NRV}, in geometrically doubling metric spaces. For the
case of spaces of homogeneous type, we refer to the papers by
Aimar et al.~\cite{Aimar2002, Aimar2007, Aimar2000}; see
also~\cite{Aimar2005, Aimar2011:2, Aimar2011} for related
results. In the present paper the context is a slightly more
general geometrically doubling quasi-metric space $(X,\rho)$,
with a positive Borel measure~$\mu$. Our approach differs from
that of the mentioned prior results, and it uses results from
general martingale theory. We provide the details, noting that
essentially the same construction can be found, for example,
in~\cite[Chapter 4]{Hytonen2009} in the case of Euclidean space
with the usual dyadic cubes and a non-doubling measure.

(2)~It is necessary to develop in the setting of product
spaces~$\widetilde{X}$ of homogeneous type careful definitions
of the continuous and dyadic versions of the function spaces we
consider. In the Euclidean setting the theory of product $\bmo$
and $H^1$ was developed by Chang and R.~Fefferman in the
continuous case~\cite{Cha,Fef,CF}, and by Bernard in the dyadic
case~\cite{Be}.

We recall that $(\XX,\rho,\mu)$ is a {\it space of homogeneous
type} in the sense of Coifman and Weiss if $\rho$~is a
quasi-metric and $\mu$ is a nonzero measure satisfying the
doubling condition. A
\emph{quasi-metric}~$\rho:\XX\times\XX\longrightarrow[0,\infty)$
satisfies (i) $\rho(x,y) = \rho(y,x) \geq 0$ for all $x$,
$y\in\XX$; (ii) $\rho(x,y) = 0$ if and only if $x = y$; and
(iii) the \emph{quasi-triangle inequality}: there is a constant
$A_0\in [1,\infty)$ such that for all $x$, $y$, $z\in\XX$, 
\begin{eqnarray}\label{eqn:quasitriangleineq}
    \rho(x,y)
    \leq A_0 [\rho(x,z) + \rho(z,y)].
\end{eqnarray}
In contrast to a metric, the quasi-metric may not be H\"older
regular and quasi-metric balls may not be open; see for
example~\cite[p.5]{HK}. A nonzero measure $\mu$ satisfies the
\emph{doubling condition} if there is a constant $C_\mu$ such
that for all $x\in\XX$ and all $r > 0$,
\begin{eqnarray}\label{doubling condition}
   \mu(B(x,2r))
   \leq C_\mu \mu(B(x,r))
   < \infty.
\end{eqnarray}
We do not impose any assumptions about regularity of the
quasi-metric, nor any additional properties of the measure.

As shown in~\cite{CW1}, spaces of homogeneous type are
geometrically doubling. Thus our construction of the Haar basis
is valid on~$(X,\rho,\mu)$.

We denote the product of such spaces by $\wX = X_1 \times
\cdots \times X_n$, with the product quasi-metric $\rho =
\rho_1 \times \cdots \times \rho_n$ and the product
measure~$\mu = \mu_1 \times \cdots \times \mu_n$.

The task of defining our function spaces on the product space
$\widetilde{X}$ is straightforward for $A_p$, $RH_p$ and
doubling weights. As in the Euclidean case, the product weights
on~$\wX$ are simply those weights that have the one-parameter
$A_p$, $RH_p$ or doubling property in each factor, and Lebesgue
measure is replaced by the Borel measure~$\mu$. The situation
is more delicate for $H^1$ and $\bmo$. For the continuous
versions, we use the definition from~\cite{HLW} of $H^1(\wX)$
via a square function that makes use of the orthonormal wavelet
bases developed in~\cite{AuHyt} for spaces of homogeneous type
in the sense of Coifman and Weiss. We define $\bmo(\wX)$ in
terms of summation conditions on the \cite{AuHyt}~wavelet
coefficients. As is shown in~\cite{HLW}, these definitions
yield the expected Hardy space theory and the duality relation
between $H^1(\widetilde{X})$ and~$\bmo(\widetilde{X})$.

For the dyadic versions of $H^1$ and $\bmo$ on~$\wX$, we
replace the \cite{AuHyt} wavelets in these definitions by a
basis of Haar wavelets, constructed in the present paper,
associated to a given system of dyadic cubes as constructed
in~\cite{HK}.


(3) In this paper we show that on product spaces~$\wX$ of
homogeneous type, in the sense of Coifman and Weiss, the space
$\bmo$ of functions of bounded mean oscillation coincides with
the intersection of finitely many dyadic $\bmo$ spaces. The
product Euclidean version of this result appears in~\cite{LPW},
the one-parameter Euclidean version in~\cite{Mei}, and the
one-parameter version for spaces~$X$ of homogeneous type
in~\cite{HK}. In addition, generalizing the Euclidean results
of~\cite{LPW}, we establish the analogous intersection results
for Muckenhoupt's $A_p$ weights $A_p(\wX)$, $1\leq p \leq
\infty$, for the reverse-H\"older weights $RH_p(\wX)$, $1\leq p
\leq \infty$, and for the class of doubling weights on~$\wX$;
the result that the Hardy space $H^1(\wX)$ is the sum of
finitely many dyadic Hardy spaces on~$\wX$; the result that the
strong maximal function on~$\wX$ is comparable to the sum of
finitely many dyadic maximal functions on~$\wX$.

We note that our methods should suffice to establish the
generalizations to spaces of homogeneous type of the weighted
dyadic structure results in~\cite{LPW}, namely Theorem~8.1 for
product $H^1$ weighted by an $A_\infty$ weight and Theorem~8.2
for the strong maximal function weighted by a doubling weight.
We leave it to the interested reader to pursue this direction.

\m

We note that in a different direction, not pursued in the
present paper, a connection between continuous and dyadic
function spaces via averaging is developed in the
papers~\cite{GJ,War,PW,Tre,PWX}, and~\cite{CLW}, for both
Euclidean spaces and spaces of homogeneous type. Specifically,
a procedure of translation-averaging (for \bmo) or
geometric-arithmetic averaging (for $A_p$, $RH_p$, and doubling
weights) converts a suitable family of functions in the dyadic
version of a function space into a single function that belongs
to the continuous version of that function space. We do not
discuss the averaging approach further in the present
paper. 

\m

This paper is organized as follows. In
Section~\ref{sec:dyadic_systems}, we first recall the
definition of systems of dyadic cubes and the related
properties of those cubes. Next we recall the result of
Hyt\"onen and the first author (see
Theorem~\ref{thm:existence2}) on the existence of a collection
of adjacent systems of dyadic cubes, which provides a
substitute in the setting of spaces of homogeneous type for the
one-third trick in the Euclidean setting. In
Section~\ref{sec:weights_and_maximal_functions} we prove our
dyadic structure results for the $A_p$, $RH_p$ and doubling
weights, as well as for maximal functions.    In
Section~\ref{sec:Haarfunctions}, we construct the Haar
functions on a geometrically doubling quasi-metric space
$(X,\rho)$ equipped with a positive Borel measure~$\mu$
(Theorem~\ref{thm:HaarFuncProp2}). Then we prove that the Haar
wavelet expansion holds on $L^p(X)$ for all $p\in(1,\infty)$
(Theorem~\ref{thm:convergence2}). In
Section~\ref{sec:functionclasses} we recall the definitions of
the continuous product Hardy and $\bmo$ spaces from~\cite{HLW},
and provide the definitions of the dyadic product Hardy and
$\bmo$ spaces by means of the Haar wavelets we have
constructed. In Section~\ref{sec:multiparameter}, we establish
the dyadic structure theorems for the continuous and dyadic
product Hardy and $\bmo$ spaces (Theorems~\ref{thm structure of
Hardy space} and~\ref{thm equivalence of BMO}), by proving the
dyadic atomic decomposition for the dyadic product Hardy spaces
(Theorem~\ref{thm atom for dyadic H1}).





\setcounter{equation}{0}




\section{Systems of dyadic cubes}
\label{sec:dyadic_systems}

The set-up for Section~\ref{sec:dyadic_systems} is a
\textit{geometrically doubling quasi-metric space}: a
quasi-metric space $(X,\rho)$ that satisfies the
\textit{geometric doubling property} that there exists a
positive integer $A_1\in \N$ such that any open ball
$B(x,r):=\{y\in X\colon \rho(x,y)<r\}$ of radius $r>0$ can be
covered by at most $A_1$ balls $B(x_i,r/2)$ of radius $r/2$; by
a quasi-metric we mean a mapping $\rho\colon X\times X\to
[0,\infty)$ that satisfies the axioms of a metric except for
the triangle inequality which is assumed in the weaker form
\[
    \rho(x,y)
    \leq A_0(\rho(x,z)+\rho(z,y))
    \quad \text{for all $x,y,z\in X$}
\]
with a constant $A_0\geq 1$.


A subset $\Omega\subseteq X$ is \emph{open} (in the topology
induced by $\rho$) if for every $x\in\Omega$ there exists
$\eps>0$ such that $B(x,\eps)\subseteq\Omega$. A subset
$F\subseteq X$ is \emph{closed} if its complement $X\setminus
F$ is open. The usual proof of the fact that $F\subseteq X$ is
closed, if and only if it contains its limit points, carries
over to the quasi-metric spaces. However, some open balls
$B(x,r)$ may fail to be open sets, see \cite[Sec 2.1]{HK}.

Constants that depend only on $A_0$ (the quasi-metric constant)
and $A_1$ (the geometric doubling constant), are referred to as
\textit{geometric constants.}

\subsection{A system of dyadic cubes}\label{sec:dyadic_cubes}
In a geometrically doubling quasi-metric space $(X,\rho)$, a
countable family
\[
    \mathscr{D}
    = \bigcup_{k\in\Z}\mathscr{D}_k, \quad
    \mathscr{D}_k
    =\{Q^k_\alpha\colon \alpha\in \mathscr{A}_k\},
\]
of Borel sets $Q^k_\alpha\subseteq X$ is called \textit{a
system of dyadic cubes with parameters} $\delta\in (0,1)$ and
$0<c_1\leq C_1<\infty$ if it has the following properties:
\begin{equation}\label{eq:cover}
    X
    = \bigcup_{\alpha\in \mathscr{A}_k} Q^k_{\alpha}
    \quad\text{(disjoint union) for all}~k\in\Z;
\end{equation}
\begin{equation}\label{eq:nested}
    \text{if }\ell\geq k\text{, then either }
        Q^{\ell}_{\beta}\subseteq Q^k_{\alpha}\text{ or }
        Q^k_{\alpha}\cap Q^{\ell}_{\beta}=\emptyset;
\end{equation}
\begin{equation}\label{eq:dyadicparent}
    \text{for each }(k,\alpha)\text{ and each } \ell\leq k,
    \text{ there exists a unique } \beta
    \text{ such that }Q^{k}_{\alpha}\subseteq Q^\ell_{\beta};
\end{equation}
\begin{equation}\label{eq:children}
\begin{split}
    & \text{for each $(k,\alpha)$ there exist between 1 and $M$
        (a fixed geometric constant)  $\beta$ such that }  \\
    & Q^{k+1}_{\beta}\subseteq Q^k_{\alpha}, \text{ and }
        Q^k_\alpha =\bigcup_{\substack{Q\in\mathscr{D}_{k+1}\\
    Q\subseteq Q^k_{\alpha}}}Q;
\end{split}
\end{equation}
\begin{equation}\label{eq:contain}
    B(x^k_{\alpha},c_1\delta^k)
    \subseteq Q^k_{\alpha}\subseteq B(x^k_{\alpha},C_1\delta^k)
    =: B(Q^k_{\alpha});
\end{equation}
\begin{equation}\label{eq:monotone}
   \text{if }\ell\geq k\text{ and }
   Q^{\ell}_{\beta}\subseteq Q^k_{\alpha}\text{, then }
   B(Q^{\ell}_{\beta})\subseteq B(Q^k_{\alpha}).
\end{equation}
The set $Q^k_\alpha$ is called a \textit{dyadic cube of
generation} $k$ with center point $x^k_\alpha\in Q^k_\alpha$
and side length~$\delta^k$. The interior and closure of
$Q^k_\alpha$ are denoted by $\widetilde{Q}^k_{\alpha}$ and
$\bar{Q}^k_{\alpha}$, respectively.

We recall from \cite{HK} the following construction, which is a
slight elaboration of seminal work by M.~Christ \cite{Chr}, as
well as Sawyer--Wheeden~\cite{SW}.

\begin{thm}\label{thm:existence}
Let $(X,\rho)$ be a geometrically doubling quasi-metric space.
Then there exists a system of dyadic cubes with parameters
$0<\delta\leq (12A_0^3)^{-1}$ and $c_1=(3A_0^2)^{-1},
C_1=2A_0$. The construction only depends on some fixed set of
countably many center points $x^k_\alpha$, having the
properties that
\begin{equation*}
    \rho(x_{\alpha}^k,x_{\beta}^k)
        \geq \delta^k\quad(\alpha\neq\beta),\qquad
    \min_{\alpha}\rho(x,x^k_{\alpha})
        < \delta^k\quad \text{for all}~x\in X,
\end{equation*}
and a certain partial order $\leq$ among their index pairs
$(k,\alpha)$. In fact, this system can be constructed in such a
way that
\begin{equation}\label{eq:closedCube}
    \overline{Q}^k_\alpha
    =\overline{\{x^{\ell}_\beta:(\ell,\beta)\leq(k,\alpha)\}}
\end{equation}
and
\begin{equation}\label{eq:openCube}
    \widetilde{Q}^k_\alpha=\operatorname{int}\overline{Q}^k_\alpha=
    \Big(\bigcup_{\gamma\neq\alpha}\overline{Q}^k_\gamma\Big)^c,
\end{equation}
and
\begin{equation}\label{eq:3cubes}
    \widetilde{Q}^k_\alpha\subseteq Q^k_\alpha\subseteq 
    \overline{Q}^k_\alpha,
\end{equation}
where $Q^k_\alpha$ are obtained from the closed sets
$\overline{Q}^k_\alpha$ and the open sets $\widetilde{Q}^k_\alpha$ by
finitely many set operations.
\end{thm}

\begin{remark}
    The proof in \cite{HK} shows that the first and the second
    inclusion in \eqref{eq:contain} hold with $Q^k_\alpha$
    replaced by $\widetilde{Q}^k_\alpha$ and 
    $\overline{Q}^k_\alpha$, respectively. 
    We mention that, for any $Q\in\mathscr{D}$, the number $M$ of
    dyadic sub-cubes as in \eqref{eq:children} is bounded by $M\leq
    A_1^2(A_0/\delta)^{\log_2A_1}$.
\end{remark}

\subsection{Further properties of dyadic cubes}
The following additional properties follow directly from the
properties listed in \eqref{eq:cover}--\eqref{eq:monotone}:
\begin{align}\label{eq:bounded}
    X\text{ is bounded if and only if there exists }
    Q\in \mathscr{D} \text{ such that } X=Q;
\end{align}
\begin{align}\label{eq:Q(x,k)}
    &\text{for every $x\in X$ and $k\in \Z$,
        there exists a unique $Q\in\mathscr{D}_k$
        such that }\\
   &x\in Q
    =: Q^k(x);\nonumber
\end{align}
\begin{align}\label{eq:isolated}
    & x\in X \text{ is an isolated point if and only if
    there exists $k\in\Z$ such that }
    \{x\}
    = Q^{\ell} (x)\\
    &\text{ for all } \ell\geq k.\nonumber
\end{align}

\subsection{Dyadic system with a distinguished center point}
The construction of dyadic cubes requires their center points
and an associated partial order be fixed \textit{a priori}.
However, if either the center points or the partial order is
not given, their existence already follows from the
assumptions; any given system of points and partial order can
be used as a starting point. Moreover, if we are allowed to
choose the center points for the cubes, the collection can be
chosen to satisfy the additional property that a fixed point
becomes a center point at \textit{all levels}:
\begin{equation}\label{eq:fixedpoint}
\begin{split}
    &\text{given a fixed point } x_0\in X, \text{ for every } k\in \Z,
        \text{ there exists }\alpha \text{ such that } \\
    & x_0
        = x^k_\alpha,\text{ the center point of }
        Q^k_\alpha\in\mathscr{D}_k.
\end{split}
\end{equation}
This property is crucial in some applications and has useful
implications, such as the following.

\begin{lemma}\label{lem:existsdyadiccube}
    Given $x, y\in X$, there exists $k\in \Z$ such that $y\in
    Q^{k}(x)$. Moreover, if $\rho(x,y)\geq\delta^k$, then
    $y\notin Q^{k+1}(x)$. In particular, if $\rho(x,y)>0$,
    there do not exist arbitrarily large indices $k$ such that
    $y\in Q^k(x)$.
\end{lemma}

\begin{proof}
Pick $k\in \Z$ such that $x,y\in B(x_0,c_1\delta^{k})$. The
first assertion follows from \eqref{eq:contain} for the
$Q^k_\alpha$ that has $x_0$ as a center point. For the second
assertion, suppose $\rho(x,y)\geq \delta^{k}$. Denote by
$x^{k+1}_\alpha$ the center point of $Q^{k+1}(x)$. Then
\begin{align*}
    \rho(y,x^{k+1}_\alpha)
    & \geq A_0^{-1}\rho(x,y)-\rho(x,x^{k+1}_\alpha)
    \geq A_0^{-1}\delta^{k}-C_1\delta^{k+1} > C_1\delta^{k+1}
\end{align*}
since $12A_0^3\delta\leq 1$ and $C_1=2A_0$, showing that
$y\notin Q^{k+1}(x)$.
\end{proof}

\begin{lemma}\label{lem:measure2}
    Suppose $\sigma$ and $\omega$ are non-trivial positive
    Borel measures on $X$, and $A\subseteq X$ is a measurable set
    with $\omega(A)>0$. Then there exists a dyadic cube $Q\in
    \mathscr{D}$ such that $\sigma (Q)>0$ and $\omega (A\cap Q)>0$.
    In particular, if $(X,\rho,\mu)$ is a quasi-metric measure
    space, $E\subseteq X$ is a set with $\mu(E)>0$, and $x\in X$,
    then there exists a dyadic cube $Q$ such that $x\in Q$ and
    $\mu(E\cap Q)>0$.
\end{lemma}

\begin{proof}
For $k\in \Z$, consider the sets $B_k:=B(x_0,c_1\delta^{-k})$
and $A_k:=A\cap B_k$. First observe that $\sigma(B_k)>0$ for
$k>k_0$ and $\omega(A_k)>0$ for $k>k_1$. Indeed,
$X=\cup_{k=1}^{\infty}B_k$ and $B_1\subseteq B_2\subseteq
\ldots$, so that $0<\sigma(X)=\lim_{k\to\infty}\sigma(B_k)$.
Similarly, $A=\cup_{k=1}^{\infty}A_k$ and $A_1\subseteq
A_2\subseteq \ldots$, so that
$0<\omega(A)=\lim_{k\to\infty}\omega(A_k)$. Set
$k=\max\{k_0,k_1\}$ and let $Q\in\mathscr{D}$ be the dyadic
cube of generation $-k$ centred at $x_0$. Then $B_k\subseteq Q$
by \eqref{eq:contain}, and it follows that $\sigma(Q)\geq
\sigma(B_k)>0$ and $\omega(A\cap Q)\geq \omega(A\cap
B_k)=\omega(A_k)>0$.
\end{proof}

We will need the following consequence
of~\eqref{eq:fixedpoint}.

\begin{lemma}\label{lem:grows_to_X}
    For any $x\in X$, $Q^k(x)\to X$ as $k\to -\infty$.
\end{lemma}

\begin{proof}
Given $x\in X$, pick $k\in \Z$ such that $x\in
B(x_0,c_1\delta^k)$ where $x_0\in X$ is as in
\eqref{eq:fixedpoint}. Then $x\in Q^k$ where $Q^k$ is the
dyadic cube in $\mathscr{D}_k$ with center point $x_0$. The
assertion follows by
\[
    Q^k\supseteq B(x_0,c_1\delta^k)\to X \text{ as } k\to -\infty.
\]
\end{proof}

\begin{remark}
    Note that the usual Euclidean dyadic cubes of the form
    $2^{-k}([0,1)^n+m), k\in\Z, m\in\Z^n$, in $\R^n$ do not
    have the property \eqref{eq:fixedpoint}, and that
    Lemmata~\ref{lem:existsdyadiccube}, \ref{lem:measure2} and
    \ref{lem:grows_to_X} depend on this property. A simple
    example of a system with a distinguished center point,
    using triadic rather than dyadic intervals, is as follows.
    In the real line divide each interval $[n, n+ 1)$,
    $n\in\Z$, into three triadic intervals of equal length,
    divide each of these intervals in three, and so on. Now
    require that the parent interval of $[0,1)$ is $[-1,2)$, so
    that $[0,1)$ is the middle third of its parent interval.
    Similarly require that the parent interval of $[-1,2)$ is
    $[-4,5)$, and so on. Then the point $x_0 = 1/2$ is a
    distinguished center point for this system.
\end{remark}

\subsection{Adjacent systems of dyadic cubes}
In a geometrically doubling quasi-metric space $(X,\rho)$, a
finite collection $\{\mathscr{D}^t\colon t=1,2,\ldots ,T\}$ of
families $\mathscr{D}^t$ is called a \textit{collection of
adjacent systems of dyadic cubes with parameters} $\delta\in
(0,1), 0<c_1\leq C_1<\infty$ and $1\leq C<\infty$ if it has the
following properties: individually, each $\mathscr{D}^t$ is a
system of dyadic cubes with parameters $\delta\in (0,1)$ and $0
< c_1 \leq C_1 < \infty$; collectively, for each ball
$B(x,r)\subseteq X$ with $\delta^{k+3}<r\leq\delta^{k+2},
k\in\Z$, there exist $t \in \{1, 2, \ldots, T\}$ and
$Q\in\mathscr{D}^t$ of generation $k$ and with center point
$^tx^k_\alpha$ such that $\rho(x,{}^tx_\alpha^k) <
2A_0\delta^{k}$ and
\begin{equation}\label{eq:ball;included}
    B(x,r)\subseteq Q\subseteq B(x,Cr).
\end{equation}

We recall from \cite{HK} the following construction.

\begin{thm}\label{thm:existence2}
    Let $(X,\rho)$ be a geometrically doubling quasi-metric space.
    Then there exists a collection $\{\mathscr{D}^t\colon
    t = 1,2,\ldots ,T\}$ of adjacent systems of dyadic cubes with
    parameters $\delta\in (0, (96A_0^6)^{-1}), c_1 = (12A_0^4)^{-1},
    C_1 = 4A_0^2$ and $C = 8A_0^3\delta^{-3}$. The center points
    $^tx^k_\alpha$ of the cubes $Q\in\mathscr{D}^t_k$ have, for each
    $t\in\{1,2,\ldots,T\}$, the two properties
    \begin{equation*}
        \rho(^tx_{\alpha}^k, {}^tx_{\beta}^k)
        \geq (4A_0^2)^{-1}\delta^k\quad(\alpha\neq\beta),\qquad
        \min_{\alpha}\rho(x,{}^tx^k_{\alpha})
        < 2A_0\delta^k\quad \text{for all}~x\in X.
    \end{equation*}
    Moreover, these adjacent systems can be constructed in such a
    way that each $\mathscr{D}^t$ satisfies the distinguished
    center point property \eqref{eq:fixedpoint}.
\end{thm}

\begin{remark}
    For $T$ (the number of the adjacent systems of dyadic
    cubes), we have the estimate
    \begin{equation}\label{eq:upperbound}
        T
        = T(A_0,A_1,\delta)
        \leq A_1^6(A_0^4/\delta)^{\log_2A_1}.
    \end{equation}
    Note that in the Euclidean space $\R^n$ with the usual
    structure we have $A_0=1, A_1\geq 2^n$ and
    $\delta=\tfrac{1}{2}$, so that \eqref{eq:upperbound} yields an
    upper bound of order $2^{7n}$. We mention that T. Mei
    \cite{Mei} has shown that in~$\R^n$ the conclusion
    \eqref{eq:ball;included} can be obtained with just $n+1$
    cleverly chosen systems~$\mathscr{D}^t$.
\end{remark}

Further, we have the following result on the smallness of the
boundary.
\begin{prop}
    Suppose that $144A_0^8\delta\leq 1$. Let $\mu$ be a
    positive $\sigma$-finite measure on $X$. Then the
    collection $\{\mathscr{D}^t\colon t=1,2,\ldots ,T\}$ may be
    chosen to have the additional property that
    \[
        \mu(\partial Q) = 0
        \quad  \textup{ for all } \; Q\in\bigcup_{t=1}^{T}\mathscr{D}^t.
    \]
\end{prop}


\section{Dyadic structure theorems for maximal functions and for weights}
\label{sec:weights_and_maximal_functions}

In this section we establish that in the setting of product
spaces of homogeneous type, the strong maximal function is
pointwise comparable to a sum of finitely many dyadic maximal
functions, and that the strong $A_p$ class is the intersection
of finitely many dyadic $A_p$ classes, and similarly for
reverse-H\"older weights $RH_p$ and for doubling measures.

\subsection{Maximal functions} \label{sec:maximal}
\setcounter{equation}{0}



Let $(X,\rho,\mu)$ be a space of homogenenous type, and let
$\{\D^t : t = 1, \ldots, T\}$ be a collection of adjacent
systems of dyadic cubes in~$X$ as in
Theorem~\ref{thm:existence}.

Similarly, for each $j\in\{1, \ldots, k\}$ let
$(X_j,\rho_j,\mu_j)$ be a space of homogeneous type, with an
associated collection of adjacent systems of dyadic cubes
$\{\D^{t_j} : t_j = 1, \ldots, T_j\}$. Let $\wX := X_1 \times
\cdots \times X_k$ with the product quasi-metric $\rho_1 \times
\cdots \times \rho_k$ and the product measure $\mu_1 \times \cdots
\times \mu_k$.


\begin{thm}\label{thm:maximal_function_homog}
    The maximal function~$M$ (or strong maximal function~$M_s$
    in the product case) controls each dyadic maximal function
    pointwise, and is itself controlled pointwise by a sum of
    dyadic maximal functions, as follows.
    \begin{enumerate}
        \item[\textup{(i)}] Let $(X,\rho,\mu)$ be a space
            of homogeneous type. Then there is a constant
            $C > 0$ such that for each $f\in
            L^1_\textup{loc}(X)$, and for all $x\in X$, we
            have the pointwise estimates
            \begin{align*}
                M_d^t f(x)
                &\leq C Mf(x)
                \qquad
                \textrm{for each $t\in\{1,\ldots,T\}$, and}\\
                Mf(x)
                &\leq C \sum_{t = 1}^T M_d^t f(x).
            \end{align*}

        \item[\textup{(ii)}] Let $(X_j,\rho_j,\mu_j)$,
            $j\in\{1,\ldots,k\}$, be spaces of homogeneous
            type. Then there is a constant $C > 0$ such
            that for each $f\in L^1_\textup{loc}(\wX)$,
            \begin{align*}
                M_d^{t_1,\ldots,t_k}f(x)
                &\leq C M_s f(x)
                \qquad
                \textrm{for each $t_j\in\{1,\ldots,T_j\}$,
                    $j\in\{1,\ldots,k\}$, and}\\
                M_s f(x)
                &\leq C \sum_{t_1 = 1}^{T_1} \cdots
                    \sum_{t_k = 1}^{T_k} M_d^{t_1,\ldots,t_k} f(x).
            \end{align*}
    \end{enumerate}
\end{thm}

These maximal operators are defined as follows. For $f\in
L^1_\text{loc}(X,\mu)$ let $Mf$ denote the Hardy--Littlewood
maximal function, given by
\[
   Mf(x)
   := \sup_{B\ni x}\frac{1}{\mu(B)} \int_B |f(y)| \, d\mu(y),
\]
where the supremum is taken over all balls $B\subset X$ that
contain~$x$. For each $t\in\{1,\ldots,T\}$, denote by $M_d^t f$
the dyadic maximal function with respect to the system $\D^t$
of dyadic cubes in~$X$; here the supremum is taken over only
those dyadic cubes $Q\in\D^t$ that contain~$x$.

In the multiparameter case, instead of the Hardy--Littlewood
maximal function we consider the strong maximal function~$M_s
f$, defined as follows. Take $x = (x_1,\ldots,x_k)\in \wX$ and
$f\in L^1_\text{loc}(\wX,\mu_1 \times \cdots \times \mu_k)$.
Let
\begin{eqnarray}\label{eqn:strong maximal}
    M_s f(x)
    := \sup_{B_1\times\cdots\times B_k\ni x}
        \frac{1}{\prod_{j = 1}^k \mu_j(B_j)}
        \int_{B_1\times\cdots\times B_k} |f(y)| \,
        d\mu_1(y_1) \times\cdots\times d\mu_k(y_k),
\end{eqnarray}
where $y = (y_1,\ldots,y_k)$ and the supremum is taken over all
products $B_1\times\cdots\times B_k$ of balls $B_j\subset X_j$,
$B_j\ni x_j$, for $j\in\{1,\ldots,k\}$.

Next, for each choice of $t_j\in\{1,\ldots,T_j\}$, for
$j\in\{1,\ldots,k\}$, let $M_d^{t_1,\ldots,t_k} f$ denote the
associated dyadic strong maximal function, defined by
\[
    M_d^{t_1,\ldots,t_k} f(x)
    := \sup_{Q_1\times\cdots\times Q_k\ni x}
        \frac{1}{\prod_{j = 1}^k \mu_j(Q_j)}
        \int_{Q_1\times\cdots\times Q_k} |f(y)| \,
        d\mu_1(y_1) \times\cdots\times d\mu_k(y_k),
\]
restricting the supremum in formula~\eqref{eqn:strong maximal}
to dyadic rectangles $Q_1\times\cdots\times
Q_k\in\D^{t_1}\times\cdots\times\D^{t_k}$ that contain~$x$.

\begin{proof}[Proof of
Theorem~\ref{thm:maximal_function_homog}] (i) Fix $x\in X$. Fix
$t\in\{1,\ldots,T\}$ and suppose $Q\ni x$, $Q\in\D^t$. Then $Q
= Q^{t,k}_\alpha$ for some $k$ and $\alpha$, and by condition
\eqref{eq:contain}, there is a ball $B =
B(x^k_\alpha,C_1\delta^k)$ that contains~$Q$. Since $\mu$ is
doubling, we see that
\[
    \frac{1}{\mu(Q)} \int_Q |f(y)| \, d\mu(y)
    \leq \frac{C^*}{\mu(B)} \int_B |f(y)| \, d\mu(y),
\]
where $C^*= C_{dbl}^{\log_2(C_1/c_1)+1}$ and $C_1,c_1$ are the
constants in \eqref{eq:contain}, see Proposition
\ref{prop:dbl}(i). It follows that $M_d^t f(x) \lesssim Mf(x)$.

For the second inequality, fix $x_0\in X$ and a ball $B(x,r)\ni
x_0$. By \eqref{eq:ball;included} there exist
$t\in\{1,\ldots,T\}$ and $Q\in\D^t$ such that $B(x,r) \subset Q
\subset B(x,Cr)$. Since $\mu$ is doubling, $\mu(B(x,Cr))
\lesssim \mu(B(x,r)) \lesssim \mu(Q)$. It follows that $Mf(x)
\lesssim \sum_{t = 1}^T M_d^t f(x)$.

(ii) Fix $x\in X$, and fix $(t_1,\ldots,t_k)\in \{1,\ldots,T_1\}
\times \cdots \times \{1,\ldots,T_k\}$. Then suppose
$Q_1\times\cdots\times Q_k \ni x$, $Q_1\times\cdots\times Q_k
\in \D^{t_1} \times \cdots \times \D^{t_k}$. Iteration of the
argument in (i), using in each factor the condition
\eqref{eq:contain} and the assumption that each $\mu_j$ is
doubling, establishes the first inequality.

Similarly, iteration of the argument for the second inequality
in~(i), using in each factor the
condition~\eqref{eq:ball;included} and the assumption that
$\mu_j$ is doubling, establishes that $M_s f(x) \lesssim
\sum_{t_1 = 1}^{T_1} \cdots \sum_{t_k = 1}^{T_k}
M_d^{t_1,\ldots,t_k} f(x)$.
\end{proof}

\subsection{Doubling, $A_p$, and $RH_p$ weights}
\label{sec:resultsApRHpdbl} \setcounter{equation}{0}

\textcolor{black}{Like in the previous section, for each $j\in\{1, \ldots, k\}$ let
$(X_j,\rho_j,\mu_j)$ be a space of homogeneous type, with an
associated collection of adjacent systems of dyadic cubes
$\{\D^{t_j} : t_j = 1, \ldots, T_j\}$. Let $\wX := X_1 \times
\cdots \times X_k$ with the product quasi-metric $\rho_1 \times
\cdots \times \rho_k$ and the product measure $\mu_1 \times \cdots
\times \mu_k$.}

By a \emph{weight} \textcolor{black}{ on~$(X,\rho,\mu)$, a space of homogeneous type,} we mean a nonnegative locally
integrable function $\om : X \to [0,\infty]$. We begin with the
main result of this subsection; definitions of doubling
weights, $A_p$ weights and $RH_p$ weights are discussed below.

\begin{thm}\label{thm:metricproduct intersectionoftranslatesApRHpdbl}
    Fix $k\in\N$. For each $j\in\{1,2,\ldots,k\}$, let
    $(X_j,\rho_j,\mu_j)$ be a space of homogeneous type, 
     and as in Theorem~\ref{thm:existence2}, let
    $\{\D^{t_j}: t_j=1,\ldots, T_j\}$ be a
    collection of adjacent systems of dyadic cubes for~$X_j$.
    Then the following assertions hold.
    \begin{enumerate}
        \item[\textup{(a)}] A weight $\om(x_1,\ldots,x_k)$
            is a product doubling weight if and only if
            $\om$ is dyadic doubling with respect to each
            of the $T_1T_2\cdots T_k$ product dyadic
            systems
            $\D_1^{t_1}\times\cdots\times\D_k^{t_k}$, where
            $t_j$ runs over $\{1, \ldots, T_j\}$ for each
            $j\in\{1, 2,\ldots,k\}$, with comparable
            constants.

        \item[\textup{(b)}] For each $p$ with $1 \leq p
            \leq \infty$, $k$-parameter
            $A_p(X_1\times\cdots\times X_k)$ is the
            intersection of $T_1T_2\cdots T_k$
            $k$-parameter dyadic $A_p$ spaces, as follows:
            \[
                A_p(X_1\times\cdots\times X_k)
                = \bigcap_{t_1 = 1}^{T_1} \bigcap_{t_2 = 1}^{T_2} \cdots
                    \bigcap_{t_k = 1}^{T_k}
                    A_{p,d}^{t_1,t_2,\ldots,t_k}(X_1\times
                    \cdots\times X_k),
            \]
            with comparable constants.



        \item[\textup{(c)}] For each $p$ with $1 \leq p
            \leq \infty$, $k$-parameter
            $RH_p(X_1\times\cdots\times X_k)$ is the
            intersection of $T_1T_2\cdots T_k$
            $k$-parameter dyadic $RH_p$ spaces, as follows:
            \[
                RH_p(X_1\times\cdots\times X_k)
                = \bigcap_{t_1 = 1}^{T_1} \bigcap_{t_2 = 1}^{T_2} \cdots
                    \bigcap_{t_k = 1}^{T_k}
                    RH_{p,d}^{t_1,t_2,\ldots,t_k}(X_1\times
                    \cdots\times X_k),
            \]
            with comparable constants.
%
%
    \end{enumerate}
    The constant $A_p(\om)$ depends only on the
    constants~$A_{p,d}^{t_1,t_2,\ldots,t_k}(\om)$ for $1 \leq
    t_j \leq T_j$, $1 \leq j \leq k$, and
    vice versa, and similarly for the other classes.
%
\end{thm}

We note that one difference from the Euclidean setting is that
on $X$ or~$\wX$, it is not immediate that a continuous function
space is a subset of its dyadic counterpart, since in general
the dyadic cubes are not balls. We address this question in the
proof of Theorem~\ref{thm:metricproduct
intersectionoftranslatesApRHpdbl}, below.

We begin with some definitions and observations. For brevity,
we include only the one-parameter versions on~$X$. The product
definitions on~$\wX$ follow the pattern described, for the
Euclidean case, in~\cite[Section~7.2]{LPW}. In particular the
product $A_p$ weights are those which are uniformly $A_p$ in
each variable separately, and so on.

\begin{defn}\label{def-dbl-dydbl-measure}
    \begin{enumerate}
    \item[(i)] A weight $\om$ on a \textcolor{black}{homogeneous space 
       ~$(X,\rho, \mu)$} is \emph{doubling} if there is a
        constant $C_\textrm{dbl}$ such that for all $x\in
        X$ and all $r > 0$,
        \[
                    0
                    < \om(B(x,2r))
                    \leq C_\textrm{dbl} \, \om(B(x,r))
                    < \infty.
        \]
        As usual, $\om(E) := \int_E \om \textcolor{black}{\,d\mu}$ for $E\subset X$.

    \item[(ii)] A weight $\om$ on a \textcolor{black}{ homogeneous 
        space~$(X,\rho, \mu)$} equipped with a system $\D$ of
        dyadic cubes is \emph{dyadic doubling} if there is
        a constant $C_\textrm{dydbl}$ such that for every
        dyadic cube $Q\in\D$ and for each child $Q'$
        of~$Q$,
        \[
                    0
                    < \om(Q)
                    \leq C_\textrm{dydbl} \, \om(Q')
                    < \infty.
        \]
    \end{enumerate}
\end{defn}

\begin{prop}\label{prop:dbl}
    \begin{enumerate}
        \item[(i)] Let $(X,\rho,\mu)$ be a space of homogeneous type.
            Then
            for all $x\in X$, $r
            > 0$, and $\lambda > 0$, we have
            \begin{equation}\label{eqn:dbllambda}
                \mu(B(x,\lambda r))
                \leq (C_\textup{dbl})^{1 + \log_2 \lambda} \,
                    \mu(B(x,r)).
            \end{equation}


        \item[(ii)] The condition in
            Definition~\ref{def-dbl-dydbl-measure}~\textup{(ii)}
            of dyadic doubling weights, above, is
            equivalent to the following condition: there is
            a constant $C$ such that for each cube
            $Q\in\D$, for all children (subcubes) $Q'$ and
            $Q''$ of $Q$,
            \begin{equation}\label{def-dydbl-measure-children}
                \frac{1}{C}\om(Q'')
                \leq \om(Q')
                \leq C\om(Q'').
            \end{equation}
    \end{enumerate}
\end{prop}

\begin{proof}
    The proofs are elementary; (ii) is an immediate consequence
    of the fact that each cube $Q$ in $\D$ is the disjoint
    union of its children (condition~\eqref{eq:children}).
\end{proof}

We note that in the Euclidean setting, the theory of product
weights was developed by K.-C. Lin in his thesis~\cite{L}, and
the dyadic theory was developed in Buckley's paper~\cite{Buc}.
The definitions of $A_p$ and $RH_p$, and their dyadic versions,
on $X$ and $\widetilde{X}$, are obtained by the natural
modifications of the Euclidean definitions, which are
summarised in, for example, \cite{PWX} for $A_p$, $1 \leq p
\leq \infty$, and $RH_p$, $1 < p \leq \infty$, and
in~\cite{LPW} for $RH_1$. For brevity, we omit the definitions
for $A_p$ and $A_{p,d}^t$, and include only those for the
reverse-H\"older classes $RH_p$ and $RH_{p,d}^t$.

\begin{defn}
  \label{def:RHp}
  Let $\omega(x)$ be a nonnegative locally integrable function
  on~\textcolor{black}{$(X,\rho, \mu)$}. For $p$ with $1 < p < \infty$, we
  say $\omega$ is a \emph{reverse-H\"older-$p$ weight}, written
  $\omega\in RH_p$, 
  if
  \[
    RH_p(\om)
    := \sup_B \left(\intav_B \omega^p\right)^{1/p}
    \left(\intav_B \omega\right)^{-1}
    < \infty.
  \]
  For $p = 1$,
  we say $\omega$ is a
  \emph{reverse-H\"older-1 weight}, written $\omega\in
  RH_1$ or $\om\in B_1$, if
  \[
    RH_1(\om)
    := \sup_B \intav_B\left(\frac{\om}{\intav_B\om} \log\frac{\om}{\intav_B\om}\right)
    < \infty.
  \]
  For $p = \infty$, we say $\omega$ is a
  \emph{reverse-H\"older-infinity weight}, written $\omega\in
  RH_\infty$ or $\om\in B_\infty$, if
  \[
    RH_\infty(\om)
    := \sup_B \left(\esssup_{x\in B} \om\right)
    \left(\intav_B \omega\right)^{-1}
    < \infty.
  \]
  Here the suprema are taken over all quasi-metric
  balls~$B\subset X$, and $\intav_B$ denotes
  ${1\over\mu(B)}\int_B\textcolor{black}{d\mu}$. The quantity $RH_p(\om)$ is called
  the \emph{$RH_p$~constant} of~$\om$.

  For $p$ with $1 \leq p \leq \infty$, and $t \in
  \{1,2,\ldots,T\}$, we say $\omega$ is a \emph{dyadic
  reverse-H\"older-$p$ weight related to the dyadic system
  $\mathcal{D}^t$}, written $\om\in RH_{p,d}^t$, 
   if
  \begin{enumerate}
  \item[(i)] the analogous condition
      $RH_{p,d}^t(\om)<\infty$
      holds with the supremum being taken over only the
      dyadic cubes $Q\subset \mathcal{D}^t$, and

  \item[(ii)] in addition $\om$ is a dyadic doubling
      weight.
  \end{enumerate}
  We define the \emph{$RH_{p,d}^t$~constant} $RH_{p,d}^t(\om)$~of~$\om$
  to be the larger of this dyadic supremum and the dyadic
  doubling constant.
\end{defn}
Note that the $A_p$ inequality (or the $RH_p$ inequality)
implies that the weight~$\om$ is doubling, and the dyadic $A_p$
inequality implies that~$\om$ is dyadic doubling. However, the
dyadic $RH_p$ inequality does not imply that $\om$ is dyadic
doubling, which is why the dyadic doubling assumption is needed
in the definition of~$RH_{p,d}^t$.

It is shown in~\cite{AHT} that the self-improving property of
reverse-H\"older weights does not hold on some spaces of
homogeneous type, but does hold on doubling metric measure
spaces with some additional geometric properties such as the
$\alpha$-annular decay property; see also~\cite{Maa}. Our work
here does not involve the self-improving property.

\begin{proof}[Proof of
    Theorem~\ref{thm:metricproduct
    intersectionoftranslatesApRHpdbl}] The multiparameter case
    follows from the one-parameter case by a straightforward
    iteration argument, as in the Euclidean setting
    (see~\cite[Theorem~7.3]{LPW}).

    The one-parameter proof that each continuous function class
    contains the intersection of its dyadic counterparts
    follows that of Theorem~7.1 in~\cite{LPW}, replacing
    Lebesgue measure by~$\mu$, and replacing the two dyadic
    grids related through translation by a collection of
    adjacent systems of dyadic grids. We give the details for
    the cases of doubling weights and $RH_p$ weights (parts~(a)
    and~(c)), and omit them for $A_p$ weights (part~(b)).

    (a) We show that doubling weights~$\om$ on~$X$ are dyadic
    doubling with respect to each of the systems $\D^t$ of
    dyadic cubes, $t = 1, \ldots, T$, given by
    Theorem~\ref{thm:existence2}; these systems have parameters
    $\delta$, $c_1=1/(12A_0^4)$, $C_1=4A_0^2$
    and~$C=8A_0^3/\delta^3$. Let $\om$ be a doubling weight
    on~$X$. Fix $t\in\{1, \ldots, T\}$. Fix a dyadic
    cube~$Q^k_\al\in\D^t$ and a child $Q^{k+1}_\beta$ of
    $Q^k_\al$. Then $B(Q^k_\al) := B(x^k_\al,C_1\delta^k)
    \subset B(x^{k + 1}_\beta, 2 A_0 C_1 \delta^k)$, since for
    $x\in B(Q^k_\al)$,
    \[
        \rho(x,x^{k + 1}_\beta)
        \leq A_0[\rho(x,x^k_\al) + \rho(x^k_\al, x^{k + 1}_\beta))]
        \leq 2 A_0 C_1 \delta^k.
    \]
    By Proposition~\ref{prop:dbl}~(i), it follows that
    \begin{align*}
        \om(Q^k_\al)
        &\leq \om(B(x^k_\al,C_1\delta^k))
        \leq \om(B(x^{k + 1}_\beta, 2 A_0 C_1 \delta^k))\\
        &= \om(B(x^{k + 1}_\beta, 2 A_0 C_1/(c_1\delta) \, c_1 \delta^{k + 1})) \\
        &\leq (C_\textrm{dbl})^{1 + \log_2 [2 A_0 C_1/(c_1\delta)]} \,
            \om(B(x^{k + 1}_\beta, c_1 \delta^{k + 1}))\\
        &\leq  (C_\textrm{dbl})^{1 + \log_2 [2 A_0 C_1/(c_1\delta)]} \,
            \om(Q^{k + 1}_\beta).
    \end{align*}
    Therefore $\om$ is dyadic doubling, with constant
    $(C_\textrm{dbl})^{1 + \log_2 [2 A_0 C_1/(c_1\delta)]}$,
    with respect to each system~$\D^t$, for $t = 1, \ldots, T$.

    For the other inclusion, let $\om$ be a weight that is
    dyadic doubling with constant $C_\textrm{dydbl}$ with
    respect to each $\D^t$, for $t = 1, \ldots, T$. Fix $x\in
    X$ and $r > 0$. Pick $k\in \Z$ such that $\delta^{k+3}<
    2r\leq \delta^{k+2}$. By property~\eqref{eq:ball;included}
    applied to $B(x,2r)$, there is some $t \in \{1, \ldots,
    T\}$ and some $Q = Q^k_\alpha \in \D^t$ of generation $k$
    such that
    \begin{equation}\label{eqn:cubeball2r}
        B(x,2r)
        \subset Q
        = Q^k_\alpha
        \subset B(x,2Cr).
    \end{equation}

    We claim there is an integer $N$ independent of $Q$ and $t$
    such that there is a descendant $Q' = Q^\ell_\beta$ of $Q$
    at most $N$ generations below $Q$, with $Q'\subset Q$,
    $Q'\in\D^t$, and $Q' \subset B(x,r)$. If so, then
    \[
        \om(B(x,r))
        \leq \om(Q)
        \leq C_\textrm{dydbl}^N \, \om(Q')
        \leq C_\textrm{dydbl}^N \, \om(B(x,r)),
    \]
    and so $\om$ is a doubling weight.

    It remains to establish the claim.
   First, choose $\ell\in\Z$ such that
    \[
        \delta^\ell
        \leq (6A_0^3)^{-1}r
        < \delta^{\ell - 1}.
    \]
    Second, by the choice of the center points for the adjacent
    systems of dyadic cubes as in Theorem~\ref{thm:existence2},
    there is a $\beta$ such that
    $\rho(x,x^\ell_\beta)<2A_0\delta^\ell$.
    It follows that $Q' :=
    Q^\ell_\beta \subset B(x,r)$. For given $y\in
    Q^\ell_\beta$, we have
    \begin{align*}
        \rho(y,x)
        &\leq A_0[\rho(y,x^\ell_\beta) + \rho(x^\ell_\beta, x)] \\
        &< A_0[C_1\delta^\ell + 2A_0\delta^\ell] \\
        &\leq   6A_0^3\delta^\ell \\
        &\leq r,
    \end{align*}
    so $y\in B(x,r)$. Hence $Q^\ell_\beta\subset B(x,r)$.


    Now we have
    \[
        \left(\frac{1}{\delta}\right)^{\ell - k}
        = \frac{\delta^k}{\delta^{\ell}}
        \leq \frac{2r/\delta^3}{\delta r/(8A_0^3)}
        = \frac{16A_0^3}{\delta^4}.
    \]
    It follows that
    \[
        \ell - k
        \leq \frac{\log[16A_0^3/\delta^4]}{\log[1/\delta]}
        < \infty.
    \]
    Thus it suffices to choose $N := \lceil
    \log[16A_0^3/\delta^4])/(\log[1/\delta])\rceil$.

    We note that with the parameter choices (as in
    Theorem~\ref{thm:existence2}) of $c_1 = 1/(12A_0^4)$, $C_1
    = 4A_0^2$, $C = 8A_0^3/\delta^3$, and choosing $\delta =
    1/(96A_0^6)$ at the upper end of the range in
    Theorem~\ref{thm:existence2}, we find that
    $16A_0^3/\delta^4=16\cdot 96^4A_0^{27}\gg 1$ while
    $1/\delta = 96A_0^2$. Thus with these parameter choices we
    have $1 < N < \infty$.


    We also note that the parameter choices in
    Theorem~\ref{thm:existence2} are consistent with those in
    Theorem~\ref{thm:existence}. This completes the proof of
    part~(a).


\smallskip
(b) The proof for $A_p$ weights is similar to part~(c) below,
and we omit the details.

\smallskip
(c) Now we turn to the reverse-H\"older weights. It suffices to
prove the one-parameter case. Suppose $1 < p < \infty$. We show
first that $RH_p \subset RH_{p,d}^t$ for each $t \in
\{1,2,\ldots, T\}$. Fix such a~$t$ and fix $\om\in RH_p$.
Consider the quantity
\begin{align}\label{ee1}
    V :=
    \left(\intav_Q \omega^p\right)^{1/p}
    \left(\intav_Q \omega\right)^{-1},
\end{align}
where $Q$ is any fixed dyadic cube in $\mathcal{D}^t$. For that
dyadic cube~$Q$, we denote by $B_{c_1}$ and $B_{C_1}$ the two
balls from property \eqref{eq:contain}.

Then, since $Q\subseteq B_{C_1}$,  we have that
\begin{align*}
    V 
    = \left({1\over\mu(Q)}\int_Q \omega^p\,\textcolor{black}{d\mu}\right)^{1/p}{\mu(Q)\over \om(Q)}
    \leq \left({1\over\mu(Q)} \int_{B_{C_1}} \omega^p\,\textcolor{black}{d\mu}\right)^{1/p}
        {\mu(B_{C_1})\over \om(Q)}.
\end{align*}
Next, since $B_{c_1}\subseteq Q$, we have $\mu(Q)\geq
\mu(B_{c_1})\geq \widetilde{C}\mu(B_{C_1})$, where the last
inequality follows from the doubling property of $\mu$, and the
constant $\widetilde{C}= \big( C_\mu^{1+\log_2{C_1\over c_1}}
\big)^{-1}$. Similarly,  we have $\om(Q)\geq \om(B_{c_1})\geq
\widetilde{\widetilde{C\,}}\om(B_{C_1})$, where the last
inequality follows from the doubling property of~$\om$, and the
constant $\widetilde{\widetilde{C\,}} = \big( C_{\rm
dbl}^{1+\log_2{C_1\over c_1}} \big)^{-1}$. Hence
\begin{align*}
    V 
    \leq  {1\over \widetilde{C}^{1\over p} \widetilde{\widetilde{C\,}}}
        \left({1\over\mu(B_{C_1})}
        \int_{B_{C_1}} \omega^p\,\textcolor{black}{d\mu}\right)^{1/p}{\mu(B_{C_1})\over \om(B_{C_1})}
    \leq  {1\over \widetilde{C}^{1\over p} \widetilde{\widetilde{C\,}}}
        \sup_B \left(\intav_B \omega^p\right)^{1/p} \left(\intav_B \omega\right)^{-1}
    = {RH_p(\om)\over \widetilde{C}^{1\over p} \widetilde{\widetilde{C\,}}},
  \end{align*}
which implies that $\om \in RH_{p,d}^t$.

Next we prove that $\bigcap_{t=1}^T RH_{p,d}^t \subset RH_p $.
Suppose $\om\in \bigcap_{t=1}^T RH_{p,d}^t$. Then for each
fixed quasi-metric ball $B(x,r)\subset X$, by
property~\eqref{eq:ball;included}, there exist an integer $k$
with $\delta^{k+3}<r\leq\delta^{k+2}$, a number
$t\in\{1,\ldots,T\}$, and a cube $Q\in\mathscr{D}^t$ of
generation $k$ and with center point $x^k_\alpha$ such that
$\rho(x,x_\alpha^k)<2A_0\delta^{k}$ and $    B(x,r)\subseteq
Q\subseteq B(x,Cr). $ Here each $\mathscr{D}^t$ is a system of
dyadic cubes with parameters $\delta\in (0,1)$ and $0<c_1\leq
C_1<\infty$. Hence, writing $B := B(x,r)$, we have
\begin{align*}
    \left(\intav_B \omega^p\right)^{1/p} \left(\intav_B \omega\right)^{-1}
    = \left({1\over\mu(B)}\int_B \omega^p\,\textcolor{black}{d\mu}\right)^{1/p}{\mu(B)\over \om(B)}
    \leq  \left({1\over\mu(B)}\int_Q \omega^p\,\textcolor{black}{d\mu}\right)^{1/p}{\mu(Q)\over \om(B)}.
 \end{align*}
Also, from the doubling property of $\mu$, we have $ \mu(B)\geq
\overline{C}\mu(B(x,Cr)) \geq \overline{C}\mu(Q) $, where
$\overline{C}= \big( C_\mu^{1+\log_2{C}} \big)^{-1}$.
Similarly,   $ \om(B)\geq \overline{\overline{C}}\om(B(x,Cr))
\geq \overline{\overline{C}}\om(Q) $, where
$\overline{\overline{C}}= \big( C_{\rm dbl}^{1+\log_2{C}}
\big)^{-1}$. As a consequence, we get
\begin{align*}
    \left(\intav_B \omega^p\right)^{1/p} \left(\intav_B \omega\right)^{-1}
    \leq  {1\over \overline{C}^{1\over p} \overline{\overline{C}}}
        \left({1\over\mu(Q)}\int_Q \omega^p\,\textcolor{black}{d\mu}\right)^{1/p}{\mu(Q)\over \om(Q)}
    \leq  {1\over \overline{C}^{1\over p} \overline{\overline{C}}} RH_{p,d}^t(\om),
 \end{align*}
which implies that $\om\in RH_p $.

Similar arguments apply to the cases $p = 1$ and $p = \infty$.
\end{proof}



\section{Explicit construction of Haar functions, and completeness}
\label{sec:Haarfunctions}

This section is devoted to our construction of a Haar basis
$\{h^Q_u\colon Q\in\D, u = 1,\dots,M_Q - 1\}$ for $L^p(X,\mu)$,
$1 < p < \infty$, associated to the dyadic cubes
$Q\in\mathscr{D}$, with the properties listed in
Theorems~\ref{thm:convergence} and~\ref{prop:HaarFuncProp}
below. Here $M_Q := \#\ch(Q) = \# \{R\in
\mathscr{D}_{k+1}\colon R\subseteq Q\}$ denotes the number of
dyadic sub-cubes (``children'') the cube $Q\in \mathscr{D}_k$
has.

\begin{thm}\label{thm:convergence}
    Let $(X,\rho)$ be a geometrically doubling quasi-metric
    space and suppose $\mu$ is a positive Borel measure on $X$
    with the property that $\mu(B) < \infty$ for all balls
    $B\subseteq X$. For $1 < p < \infty$, for each $f\in
    L^p(X,\mu)$, we have
    \[
        f(x)
        = m_X(f) + \sum_{Q\in\mathscr{D}}\sum_{u=1}^{M_Q-1}
            \langle f,h^Q_u\rangle h^Q_u(x), 
    \]
    where the sum converges (unconditionally) both in the
    $L^p(X,\mu)$-norm and pointwise $\mu$-almost everywhere,
    and
    \begin{equation*}
        m_X(f)
        := \left\{ \begin{array}{ll}
            \frac{1}{\mu(X)}\int_{X}fd\mu,  & \text{ if } \mu(X)<\infty,\\
            0,                              & \text{ if } \mu(X)=\infty.
        \end{array}\right.
    \end{equation*}
\end{thm}

The following theorem collects several basic properties of the
functions $h^Q_u$.

\begin{thm}\label{prop:HaarFuncProp}
    The Haar functions $h^Q_u$, $Q\in\mathscr{D}$,
    $u = 1,\ldots,M_Q - 1$, have the following properties:
    \begin{itemize}
        \item[(i)] $h^Q_u$ is a simple Borel-measurable
            real function on $X$;
        \item[(ii)] $h^Q_u$ is supported on $Q$;
        \item[(iii)] $h^Q_u$ is constant on each
            $R\in\ch(Q)$;
        \item[(iv)] $\int h^Q_u\, d\mu = 0$ (cancellation);
        \item[(v)] $\langle h^Q_u,h^Q_{u'}\rangle = 0$ for
            $u\neq u'$, $u$, $u'\in\{1, \ldots, M_Q - 1\}$;
        \item[(vi)] the collection
            \[
                \big\{\mu(Q)^{-1/2}1_Q\big\}
                \cup \{h^Q_u : u = 1, \ldots, M_Q - 1\}
            \]
            is an orthogonal basis for the vector
            space~$V(Q)$ of all functions on $Q$ that are
            constant on each sub-cube $R\in\ch(Q)$;
        \item[(vii)] 
        if $h^Q_u\not\equiv 0$ then
            \[
                \Norm{h^Q_u}{L^p(X,\mu)}
                \simeq \mu(Q_u)^{\frac{1}{p} - \frac{1}{2}}
                \quad \text{for}~1 \leq p \leq \infty;
            \]
    \end{itemize}
    and
    \begin{itemize}
        \item[(viii)] 
                \hspace{4cm}
                $\Norm{h^Q_u}{L^1(X,\mu)}\cdot
                \Norm{h^Q_u}{L^\infty(X,\mu)} \simeq 1$.
    \end{itemize}
\end{thm}

In the remainder of this section, we develop the Haar functions
and their properties and establish
Theorems~\ref{thm:convergence} and~\ref{prop:HaarFuncProp}. The
section is organized as follows: Section~\ref{sec:set-up}
describes our assumptions on the underlying space $(X,\rho)$
and measure~$\mu$,
Section~\ref{sec:sigmaalgebrascondexpectations} presents the
dyadic $\sigma$-algebras and conditional expectations we use,
Section~\ref{sec:indexing} provides our indexing of the
sub-cubes of a given dyadic cube, and
Section~\ref{sec:martingale} develops the martingale difference
decomposition that leads to our explicit definition of the Haar
functions. In Section~\ref{sec:martingale} we also deal with
the situation when a dyadic cube has only one child, and sum up
our Haar function definitions and results in
Theorems~\ref{thm:HaarFuncProp2} and~\ref{thm:convergence2},
which complete the proof of Theorems~\ref{thm:convergence}
and~\ref{prop:HaarFuncProp}.

%

\subsection{Set-up}\label{sec:set-up}
The set-up for Section~\ref{sec:Haarfunctions} is a
geometrically doubling quasi-metric space $(X,\rho)$ equipped
with a positive Borel measure $\mu$. We assume that the
$\sigma$-algebra of measurable sets $\mathscr{F}$ contains all
balls $B\subseteq X$ with $\mu(B)<\infty$. This implies that
$\mu$ is $\sigma$-finite, in other words, the sub-collection
\[
    \mathscr{F}^0
    :=\{F\in\mathscr{F}\colon \mu(F)<\infty \}
\]
contains a countable cover: there is a collection of at most
countably many sets $F_1$, $F_2, \ldots$ in~$\mathscr{F}^0$
such that
\[
    X
    = \bigcup_{i=1}^{\infty}F_i.
\]
The space $(X,\rho,\mu)$ is called a \textit{geometrically
doubling quasi-metric measure space}.

We emphasize that in this section
(Section~\ref{sec:Haarfunctions}) we do not assume that the
measure $\mu$ is doubling. That assumption is only needed in
Sections \ref{sec:functionclasses} and~\ref{sec:multiparameter}
below.

\subsection{Dyadic $\sigma$-algebras and conditional expectations}
\label{sec:sigmaalgebrascondexpectations} Let
$\mathscr{D}=\cup_{k\in \Z}\mathscr{D}_k$ be a fixed system of
dyadic cubes with the additional distinguished center point
property \eqref{eq:fixedpoint}. Let
$\mathscr{F}_k:=\sigma(\mathscr{D}_k)$ be the (dyadic)
$\sigma$-algebra generated by the countable partition
$\mathscr{D}_k$. It is an easy exercise to check that
\[
    \mathscr{F}_k
    = \left\{\bigcup_{\alpha\in I}Q^k_\alpha\colon I
        \subseteq \mathscr{A}_k \right\},
\]
the collection of all unions.

\begin{lemma}[Properties of the filtration $(\mathscr{F}_k)$]$\;$
    The family $(\mathscr{F}_k)$ of $\sigma$-algebras is a
    filtration, that is,
    $\mathscr{F}_i\subseteq\mathscr{F}_j\subseteq\mathscr{F}$
    for all $i<j$. Each $(X,\mathscr{F}_k,\mu)$ is
    $\sigma$-finite and
    \[
        \sigma\Big(\bigcup_{k\in\Z} \mathscr{F}_k\Big)
        = \mathscr{F}.
    \]
\end{lemma}

\begin{proof}
The first assertion is clear since every dyadic cube is a
finite union of smaller dyadic cubes. The collection
$\mathscr{D}_k\subseteq\mathscr{F}_k$ forms a countable cover,
and the second assertion follows immediately from the
assumption $\mu(B)<\infty$ imposed on balls. Recall that dyadic
cubes are Borel sets. Thus, $\cup_{k\in\Z}
\mathscr{F}_k\subseteq \mathscr{F}$ and consequently,
$\sigma(\cup_{k\in\Z} \mathscr{F}_k)\subseteq \mathscr{F}$.
Then suppose that $\Omega\in \mathscr{F}$. We may assume that
$\Omega$ is an open set.
For each $x\in\Omega$, there is $Q=Q_x\in\mathscr{D}$ such that
$x\in Q\subseteq\Omega$. Consequently, $\Omega$ is a
(countable) union of dyadic cubes, and hence, $\Omega\in
\sigma(\cup_{k\in\Z} \mathscr{F}_k)$.
\end{proof}

The proof for the following martingale convergence theorem can
be found, for example, in the lecture notes~\cite{Hytonen2009}.
\begin{thm}[Martingale convergence]
    Let $(X,\rho,\mu)$ be a geometrically doubling quasi-metric
    measure space. Suppose that $(\mathscr{F}_k)_{k\in\Z}$ is
    any filtration such that the spaces $(X,\mathscr{F}_k,\mu)$
    are $\sigma$-finite and
    \[
        \sigma\left(\bigcup_{k\in\Z}\mathscr{F}_k\right)
        = \mathscr{F}.
    \]
    Then for every $f\in L^p(\mathscr{F},\mu), 1< p<\infty$, there
    holds
    \[
        \E[f\vert \mathscr{F}_k]\to f \text{ as } k\to\infty.
    \]
    The convergence takes place both in the $L^p(X,\mu)$-norm
    and pointwise $\mu$-a.e.
\end{thm}

We consider the filtration $(\mathscr{F}_k)_{k\in \Z}$
generated by the dyadic cubes, the associated conditional
expectation operators, and the corresponding martingale
differences. For  $k\in\Z$ and $Q\in\mathscr{D}_k$, the
conditional expectation and its local version are defined for
$f\in L^1_{\loc}(X,\mathscr{F})$ by
\[
    \E_kf
    := \E[f\vert \mathscr{F}_k]
    \quad \text{and} \quad \E_Qf:=1_Q\E_kf,
\]
and they admit the explicit representation
\[
    \E_kf
    = \sum_{Q\in\mathscr{D}_k}\frac{1_Q}{\mu(Q)}\int_Qfd\mu
    = \sum_{Q\in\mathscr{D}_k}\E_Q f.
\]
If $\mu(Q) = 0$ for some cube, then the corresponding term in
the above series is taken to be $0$. We use the shorthand
notation
\[
    \langle f\rangle_Q
    := \frac{1}{\mu(Q)}\int_{Q}f\, d\mu
\]
for the integral average of $f$ on $Q$. With this notation,
$\E_Q f = 1_Q\langle f\rangle_Q$.

The martingale difference operators and their local versions
are defined by
\[
    \Db_k f
    := \E_{k+1}f-\E_kf , \quad
    \Db_Q f
    := 1_Q\Db_k f.
\]
By martingale convergence, for each $m\in \Z$,
\begin{equation}\label{eq:representationf}
    f
    =\sum_{k\geq m} \Db_k f +\E_m f
    =\sum_{k\geq m}\sum_{Q\in\mathscr{D}_k} \Db_Q f
        + \sum_{Q\in\mathscr{D}_m}\E_Q f.
\end{equation}


\subsection{The indexing of the sub-cubes}\label{sec:indexing}
Recall that a classical $L^2$-normalized Haar function $h_Q$ in
$\R^n$, associated to a standard Euclidean dyadic cube $Q$,
satisfies the size conditions $\Norm{h_Q}{L^1}=\abs{Q}^{1/2}$
and $\Norm{h_Q}{L^\infty}=\abs{Q}^{-1/2}$. In the present
context, for the cancellative Haar functions there is no upper
bound for the norm $\Norm{h^Q_u}{L^\infty}$ in terms of the
measure $\mu(Q)$. However, this will be compensated for by the
smallness of the norm $\Norm{h^Q_u}{L^1}$, in the sense that
the following estimate still holds:
\[
    \Norm{h^Q_u}{L^1}\Norm{h^Q_u}{L^\infty}
    \lesssim 1.
\]
(This type of estimate is established in~\cite{LSMP} for Haar
functions on~$\R$.) To obtain this control, we introduce the
following ordering of the sub-cubes of~$Q$.

Let the index sets $\mathscr{A}_k$, indexing the cubes
$Q^k_\alpha$ of generation $k$, be initial intervals (finite or
infinite) in $\N$. Recall from \eqref{eq:children} that for
$Q\in \mathscr{D}_k$, the cardinality of the set of dyadic
sub-cubes $\ch(Q):= \{R\in \mathscr{D}_{k+1}\colon R\subseteq
Q\}$ is bounded by
\[
    \# \ch(Q)
    =: M_Q \in [1,M];\quad
    M = M(A_0,A_1,\delta)
    < \infty .
\]
The number $M_Q$ depends, of course, on $Q$ but we may omit
this dependence in the notation whenever it is clear from the
context. For the usual dyadic cubes in the Euclidean space
$\R^n$ we have $M_Q = 2^n$ for every $Q$. We order the
sub-cubes from the ``smallest''  to the ``largest''. More
precisely, we have the following result.

\begin{lemma}\label{lem:indexing}
    Given $Q\in\mathscr{D}$, there is an indexing of the sub-cubes
    $Q_j\in \ch(Q)$ such that
    \begin{equation}\label{eq:indexing}
        \sum_{j=u}^{M_Q}\mu(Q_j)
        \geq [1-(u-1)M_Q^{-1}]\mu(Q)
    \end{equation}
    for every $u = 1$, $2$, \ldots, $M_Q$.
\end{lemma}

\begin{proof}
The case $u = 1$ is clear for any ordering of the sub-cubes.
For $u > 1$, fix some indexing. If $M_Q=1$, we are done.
Otherwise, we proceed as follows. Suppose we have an indexing
of the sub-cubes $Q_1,\ldots ,Q_{m}, m\geq 1$, such that
\eqref{eq:indexing} holds for all $u=1,\ldots, m<M_Q$. In
particular,
\begin{equation*}
    \begin{split}
    [1-(m-1)M_Q^{-1}]\mu(Q) &\leq  \sum_{j=m}^{M_Q}\mu(Q_j)
    = \frac{1}{M_Q-m} \sum_{j=m}^{M_Q}
        \sum_{\substack{\ell=m\\ \ell\neq j}}^{M_Q}\mu(Q_\ell) \\
    & \leq \frac{M_Q-(m-1)}{M_Q-m}\max_{j\in\{m,\ldots,M_Q\} } \mu
        \Big(\big(\bigcup_{\ell=m}^{M_Q}\mu(Q_\ell)\big)\setminus Q_j \Big).
    \end{split}
\end{equation*}
Thus, we may choose $n\in\{m,\ldots,M_Q\}$ such that
\[
    \mu\Big(\big(\bigcup_{\ell=m}^{M_Q}\mu(Q_\ell)\big)\setminus Q_n \Big)
    \geq \frac{M_Q-m}{M_Q-(m-1)}\cdot [1-(m-1)M_Q^{-1}]\mu(Q)
    = (1-mM_Q^{-1})\mu(Q).
\]
We reorder the cubes $Q_j$ with $j\geq m$ by setting $m = n$,
obtaining an indexing such that inequality~\eqref{eq:indexing}
holds for all $u = 1$, \ldots, $m+1$. If $M_Q > m+1$, we
proceed as before. The claim follows.
\end{proof}

We observe that one way to obtain an indexing satisfying
inequality~\eqref{eq:indexing} of Lemma~\ref{lem:indexing} is
to order the cubes so that $\mu(Q_1) \leq \mu(Q_2) \leq \cdots
\leq \mu(Q_{M_Q})$, as a short calculation shows.


From now on, let the indexing of the sub-cubes $Q_j$ of a given
cube $Q\in \mathscr{D}$ be one provided by
Lemma~\ref{lem:indexing}. For each $u\in\{1,\ldots,M_Q\}$, set
\begin{equation}\label{def:setsEk}
    E_u
    = E_u(Q)
    := \bigcup_{j=u}^{M_Q}Q_j.
\end{equation}
Then, in particular, $E_1 = Q, E_{M_Q} = Q_{M_Q}$, and $E_u =
Q_u\cup E_{u+1}$ where the union is disjoint.

\begin{lemma}\label{lem:ordering}
    For every $u\in\{1,\ldots ,M_Q-1\}$,
    \begin{equation}\label{eq:ordering1}
        \frac{\mu(Q)}{M}\leq \mu(E_u)\leq \mu(Q)
    \end{equation}
    and
    \begin{equation}\label{eq:ordering2}
        \frac{1}{M}\leq \frac{\mu(E_{u+1})}{\mu(E_u)}\leq \frac{M}{2}
    \end{equation}
    where $M$ is the uniform bound for the maximal number of
    sub-cubes. In particular, $\mu(E_u)\simeq \mu(Q)$ and
    $\mu(E_{u+1})/\mu(E_u)\simeq 1$ for all $u$.
\end{lemma}

\begin{proof}
The first estimate in \eqref{eq:ordering1} follows directly
from Lemma~\ref{eq:indexing}, and the second estimate is
immediate. For \eqref{eq:ordering2}, by
Lemma~\ref{eq:indexing},
    \begin{align*}
    \frac{1}{M}
    & \leq \frac{1}{M_Q}\leq \frac{M_Q-u}{M_Q}
    = \frac{[1-uM_Q^{-1}]\mu(Q)}{\mu(Q)}
    \leq \mu(E_u)^{-1}\sum_{j=u+1}^{M_Q}\mu(Q_j)
    = \frac{\mu(E_{u+1})}{\mu(E_u)}\\
    & \leq\left(\sum_{j=u}^{M_Q}\mu(Q_j)\right)^{-1}\mu(Q)
    \leq \frac{1}{[1-(u-1)M_Q^{-1}]}
    = \frac{M_Q}{M_Q-(u-1)}
    \leq \frac{M_Q}{2}
    \leq \frac{M}{2}.\qedhere
    \end{align*}
\end{proof}

\subsection{Martingale difference
decomposition}\label{sec:martingale}
We will decompose the operators $\E_k$ and $\Db_k$ by
representing the projections $\E_Q$ and $\Db_Q$ in terms of
Haar functions as
\begin{equation}\label{eq:projections}
    \E_Q f
    = \langle f, h^Q_0\rangle h^Q_0,\quad
    \Db_Q f
    = \sum_{u=1}^{M_Q-1} \langle f, h^Q_u\rangle h^Q_u.
\end{equation}
Here $h^Q_0:= \mu(Q)^{-1/2}1_Q$ is a non-cancellative Haar
function and $\{h^Q_u\colon u=1,\ldots,M_Q-1\}$ are
cancellative ones. To this end, given $Q\in\mathscr{D}_k$, let
the indexing of the sub-cubes $\{Q_j\colon j=1,\ldots,M_Q\}$ be
the one provided by Lemma~\ref{lem:indexing}.  Generalize the
notation $\E_Q$ by denoting
\[
    \E_Af
    := \frac{1_A}{\mu(A)}\int_{A}f \, d\mu
\]
for any measurable set $A$ with $\mu(A)>0$. With this notation
we obtain the splitting of the martingale difference $\Db_Q$ as
\begin{align*}
    \Db_Q
    &= (\E_{k+1}-\E_k)1_Q
        = \sum_{u=1}^{M}\E_{Q_u}-\E_Q
        = \sum_{u=1}^{M-1}\E_{Q_u}+(\E_{Q_M}-\E_Q)\\
    &= \sum_{u=1}^{M-1}\E_{Q_u}+\sum_{u=1}^{M-1}(\E_{E_{u+1}}-\E_{E_u})
        = \sum_{u=1}^{M-1}\big(\E_{Q_u}+\E_{E_{u+1}}-\E_{E_u} \big)
        =: \sum_{u=1}^{M-1} \Db^Q_{u};
\end{align*}
here $E_u$ is the set defined in \eqref{def:setsEk} with
$E_1=Q$ and $E_M=Q_M$.

Take a closer look at the operator $\Db^Q_u$. If $\mu(Q_u) =
0$, then $\Db^Q_uf = 0$, and we define the corresponding Haar
function to be $h^Q_u \equiv 0$. For $\mu(Q_u) > 0$, we write
(recall that $E_u = Q_u\cup E_{u+1}$ with a disjoint union)
\begin{align*}
    \Db^Q_u f & =\big(\E_{Q_u}+\E_{E_{u+1}}-\E_{E_u} \big)f\\
    & =\frac{1_{Q_u}}{\mu(Q_u)}\int_{Q_u}fd\mu + \frac{1_{E_{u+1}}}{\mu(E_{u+1})}\int_{E_{u+1}}fd\mu
    -\frac{1_{E_u}}{\mu(E_u)}\int_{E_u}fd\mu\\
    & =\frac{1_{Q_u}}{\mu(Q_u)}\int_{Q_u}fd\mu + \frac{1_{E_{u+1}}}{\mu(E_{u+1})}\int_{E_{u+1}}fd\mu
    -\frac{1_{Q_u}+1_{E_{u+1}}}{\mu(E_u)}\left(\int_{Q_u}fd\mu+\int_{E_{u+1}}fd\mu\right)\\
    &=\frac{\mu(E_{u+1})\mu(Q_u)}{\mu(E_u)} \left[
    \frac{1_{Q_u}}{\mu(Q_u)}\int \frac{1_{Q_u}}{\mu(Q_u)}fd\mu
    +\frac{1_{E_{u+1}}}{\mu(E_{u+1})}\int \frac{1_{E_{u+1}}}{\mu(E_{u+1})}fd\mu\right. \\
    & \qquad\left. {}-\frac{1_{Q_u}}{\mu(Q_u)}\int \frac{1_{E_{u+1}}}{\mu(E_{u+1})}fd\mu
    -\frac{1_{E_{u+1}}}{\mu(E_{u+1})}\int \frac{1_{Q_u}}{\mu(Q_u)}fd\mu \right]\\
    & = \frac{\mu(E_{u+1})\mu(Q_u)}{\mu(E_u)} \left(
    \frac{1_{Q_{u}}}{\mu(Q_{u})}-\frac{1_{E_{u+1}}}{\mu(E_{u+1})}\right)
    \int \left( \frac{1_{Q_{u}}}{\mu(Q_{u})}-\frac{1_{E_{u+1}}}{\mu(E_{u+1})}\right)fd\mu\\
    & =: h^Q_u\int h^Q_u fd\mu = \langle f, h^Q_u\rangle h^Q_u,
\end{align*}
where
\begin{equation}\label{def:HaarFunc}
    h^Q_u
    := a_u1_{Q_u}-b_u1_{E_{u+1}};\quad
    a_u
    := \frac{\mu(E_{u+1})^{1/2}}{\mu(Q_u)^{1/2}\mu(E_u)^{1/2}},
    b_u
    := \frac{\mu(Q_{u})^{1/2}}{\mu(E_u)^{1/2}\mu(E_{u+1})^{1/2}}.
\end{equation}

\begin{remark}
We note that it may happen that a given cube $Q$ has only one
child $R$, so that $R = Q$ as sets. In this case the
formula~\eqref{def:HaarFunc} is not meaningful but also not
relevant, as we do not need to add a cancellative Haar function
corresponding to an ``only child''~$R$. We examine this
situation more closely. A cube can be its own only child for
finitely many generations, or even forever. In the first case,
it means there will not be cancellative Haar functions
associated to the cube until it truly subdivides. In the second
case, it turns out that the cube $Q$ must be a single point,
and the point must be an isolated point. In a geometrically
doubling metric space there are at most countably many isolated
points, and the only ones that contribute something are the
ones that have positive (finite) measure (that is,
point-masses).

Given a cube that is a point-mass, there are two scenarios:
either the cube is its own parent forever, or else the cube is
a proper child of its parent. In the first case, X is a point
mass with finite measure, and in that case we need to include
the function $1_X/ \mu(X)^{1/2}$ to get a basis; in fact that
is the only element in the basis (it is an orthonormal basis,
although the orthogonality holds by default because there are
no pairs to be checked). In the second case, we don't need to
add any functions to the basis. In fact, one may then wonder
whether the characteristic function of the point-mass set can
be represented with the Haar functions we already have. The
answer is yes: we can always represent the characteristic
function of any given cube using the Haar functions associated
to ancestors of the given cube, much as on~$\R$ one can
represent the function $1_{[0,1)}$ in terms of the Haar
functions corresponding to the intervals $[0,2^n)$ for $n\geq
1$; see Project~9.7 in~\cite{PerW}. For all this and more,
see~\cite{Wei}.
\end{remark}


\begin{thm}\label{thm:HaarFuncProp2}
    Let $(X,\rho)$ be a geometrically doubling quasi-metric
    space and suppose $\mu$ is a positive Borel measure on~$X$
    with the property that $\mu(B) < \infty$ for all balls
    $B\subseteq X$. For each $Q\in\D$, let
    \[
        h^Q_0
        := \mu(Q)^{-1/2}1_Q
    \]
    and for $u = 1$, $2$, \ldots, $M_Q - 1$ let
    \[
        h^Q_u
        :=  \begin{cases}
                0, & \text{if $\mu(Q_u) = 0$}; \\
                a_u 1_{Q_u} - b_u 1_{E_{u + 1}}, & \text{if $\mu(Q_u) > 0$}, \\
            \end{cases}
    \]
    where
    \[
        a_u
        := \frac{\mu(E_{u+1})^{1/2}}{\mu(Q_u)^{1/2}\mu(E_u)^{1/2}},
        \qquad
        b_u
        := \frac{\mu(Q_{u})^{1/2}}{\mu(E_u)^{1/2}\mu(E_{u+1})^{1/2}}.
    \]
    The Haar functions $h^Q_u$, $Q\in\mathscr{D}$, $u =
    0$, $1$, \ldots, $M_Q - 1$, have the following properties:
    \begin{itemize}
        \item[(i)] each $h^Q_u$ is a simple
            Borel-measurable real function on $X$;

        \item[(ii)] each $h^Q_u$ is supported on $Q$;

        \item[(iii)] each $h^Q_u$ is constant on each $R\in
            \ch (Q)$;

        \item[(iv)] $\int h^Q_u\, d\mu = 0$ for $u = 1$,
            $2$, \ldots, $M_Q - 1$ (cancellation);

        \item[(v)] $\langle h^Q_u,h^Q_{u'}\rangle = 0$ if
            $u\neq u'$, $u$, $u'\in \{0, 1, \ldots, M_Q -
            1\}$;

        \item[(vi)] the collection
            \[
                \{h^Q_u : u = 0, 1, \ldots, M_Q - 1\}
            \]
            is an orthogonal basis for the vector
            space~$V(Q)$ of all functions on $Q$ that are
            constant on each sub-cube $R\in\ch(Q)$;

        \item[(vii)] for $u = 1$, $2$, \ldots, $M_Q - 1$,
            if $h^Q_u\not\equiv 0$ then
            \[
                \Norm{h^Q_u}{L^p(X,\mu)}
                \simeq \mu(Q_u)^{\frac{1}{p} - \frac{1}{2}}
                \quad \text{for}~1 \leq p \leq \infty;
            \]
    \end{itemize}
    and
    \begin{itemize}
        \item[(viii)] for $u = 0$, $1$, \ldots, $M_Q - 1$,
            we have
            \[
                \Norm{h^Q_u}{L^1(X,\mu)}\cdot
                \Norm{h^Q_u}{L^\infty(X,\mu)} \simeq 1.
            \]
    \end{itemize}
\end{thm}

Note that Theorem~\ref{thm:HaarFuncProp2} completes the proof
of Theorem~\ref{prop:HaarFuncProp}.

\begin{proof}[Proof of Theorem~\ref{thm:HaarFuncProp2}]
Properties (i)--(v) are clear from the definition of the Haar
functions. For property~(vi), observe that $Q$ has $M_Q$
children, and $V(Q)$ is a finite-dimensional vector space with
dimension $\dim(V(Q)) = M_Q$. The functions $\{h^Q_u\}_{u =
0}^{M_Q - 1}$ are orthogonal, and there are $M_Q$ of them, so
they span~$V(Q)$.

Property~(vii): Take $h^Q_u$ with $u\in\{1, 2, \ldots, M_Q -
1\}$. For $1 \leq p < \infty$, we have
\begin{align*}
    \|h^Q_u\|_{L^p(X,\mu)}^p
    &= \int_{Q_u} |h^Q_u(x)|^p \, d\mu(x) + \int_{E_{u + 1}} |h^Q_u(x)|^p \, d\mu(x) \\
    &= \int_{Q_u} (a_u)^p \, d\mu(x) + \int_{E_{u + 1}} (b_u)^p \, d\mu(x) \\
    &= \underbrace{\left(\frac{\mu(E_{u + 1})}{\mu(E_u)}\right)^{p/2}\frac{\mu(Q_u)}{\mu(Q_u)^{p/2}}}_{\textup{(I)}}
        + \underbrace{\frac{\mu(Q_u)^{p/2}}{\mu(E_u)^{p/2} \mu(E_{u + 1})^{p/2} } \, \mu(E_{u + 1})}_{\textup{(II)}}.
\end{align*}
By Lemma~\ref{lem:ordering},
\[
    \textup{(I)}
    \simeq \mu(Q_u)^{1 - p/2}
\]
and
\[
    \textup{(II)}
    \simeq \mu(Q_u)^{p/2} \, \frac{\mu(E_{u + 1})}{\mu(E_{u + 1})^p}
    \simeq \frac{\mu(Q_u)^{p/2}}{\mu(Q)^{p - 1}}
    \lesssim \frac{\mu(Q_u)^{p/2}}{\mu(Q_u)^{p - 1}}
    = \mu(Q_u)^{1 - p/2}.
\]
Since also $\textup{(II)} \geq 0$, it follows that
\[
    \mu(Q_u)^{1 - p/2}
    \lesssim \textup{(I)}
    \leq \textup{(I)} + \textup{(II)}
    = \|h^Q_u\|_{L^p(X,\mu)}^p
    \lesssim \mu(Q_u)^{1 - p/2},
\]
and so
\[
    \|h^Q_u\|_{L^p(X,\mu)}
    \simeq \mu(Q_u)^{\frac{1}{p} - \frac{1}{2}}
\]
as required.

Now suppose $p = \infty$. Note that by
Lemma~\ref{lem:ordering}, $a_u \simeq \mu(Q_u)^{-1/2}$ and
\[
    b_u
    \simeq \frac{\mu(Q_u)^{1/2}}{\mu(E_u)}
    \simeq \frac{\mu(Q_u)^{1/2}}{\mu(Q)}
    \leq \frac{\mu(Q_u)^{1/2}}{\mu(Q_u)}
    = \mu(Q_u)^{-1/2},
\]
as $0 \leq \mu(Q_u) \leq \mu(Q)$. Therefore, for all $x$ we
have
\[
    |h^Q_u(x)|
    \leq a_u 1_{Q_u}(x) + b_u 1_{E_{u + 1}}(x)
    \lesssim \mu(Q_u)^{-1/2}.
\]
It follows that $\|h^Q_u\|_{L^\infty(X,\mu)} \lesssim
\mu(Q_u)^{-1/2}$. Also, if $h^Q_u\not\equiv 0$, then for $x\in
Q_u$ we have $|h^Q_u(x)| = a_u \simeq \mu(Q_u)^{-1/2}$, while
for $x\in E_{u + 1}$ we have $|h^Q_u(x)| = b_u \lesssim
\mu(Q_u)^{-1/2}$. Then $\|h^Q_u\|_{L^\infty(X,\mu)} \gtrsim
\mu(Q_u)^{-1/2}$. Thus for $h^Q_u\not\equiv 0$ we find that
\[
    \|h\|_{L^\infty(X,\mu)}
    \sim
    \mu(Q_u)^{-1/2},
\]
as required.

As an aside, we note that although $|h^Q_u(x)| \lesssim
\mu(Q_u)^{-1/2}$ for all $x$, the reverse inequality need not
hold, since if the measure $\mu$ is not doubling, for $x\in
E_{u + 1}$ the quantity $|h^Q_u(x)| = b_u \sim
\mu(Q_u)^{1/2}/\mu(Q)$ may be arbitrarily small compared with
$\mu(Q_u)^{-1/2}$.

Property (viii): For $u = 0$, it is immediate from the
definition that $\|h^Q_0\|_{L^1(X,\mu)}
\|h^Q_0\|_{L^\infty(X,\mu)} = \mu(Q)^{1/2} \mu(Q)^{-1/2} = 1$.
For $u\in\{1, \ldots, M_Q - 1\}$, by property~(vii) we have
\begin{align*}
    \|h^Q_u\|_{L^1(X,\mu)} \|h^Q_u\|_{L^\infty(X,\mu)}
    &\simeq \mu(Q_u)^{1 - 1/2} \mu(Q_u)^{-1/2}
    = 1,
\end{align*}
as required.
\end{proof}

We collect the convergence results of this section in the
following theorem.
\begin{thm}\label{thm:convergence2}
    Let $(X,\rho)$ be a geometrically doubling quasi-metric
    space and suppose $\mu$ is a positive Borel measure on $X$
    with the property that $\mu(B) < \infty$ for all balls
    $B\subseteq X$. For $1 < p < \infty$, for each $f\in
    L^p(X,\mu)$, and for each $m\in\Z$, we have
    \[
        f(x)
        = \sum_{k\geq m}\sum_{Q\in\mathscr{D}_k}\sum_{u=1}^{M_Q-1}
            \langle f,h^Q_u\rangle h^Q_u(x)
            + \sum_{Q\in\mathscr{D}_m}
            \langle f,h^Q_0\rangle h^Q_0(x)
    \]
    where the first sum converges both in the $L^p(X,\mu)$-norm
    and pointwise $\mu$-a.e. For the second sum we have
    \begin{equation*}
        \sum_{Q\in\mathscr{D}_m}
        \langle f,h^Q_0\rangle h^Q_0(x) \to\left\{
        \begin{array}{ll}
            0,                              & \text{ as $m\to -\infty$ if $\mu(X)=\infty$};\\
            \frac{1}{\mu(X)}\int_{X}fd\mu , & \text{ as $m\to -\infty$ if $\mu(X)<\infty$}.
        \end{array}\right.
    \end{equation*}
    As a consequence,
    \[
        f(x)
        = m_X(f) + \sum_{Q\in\mathscr{D}}\sum_{u=1}^{M_Q-1}
            \langle f,h^Q_u\rangle h^Q_u(x).
    \]
\end{thm}

Note that Theorem~\ref{thm:convergence2} completes the proof of
Theorem~\ref{thm:convergence}.

\begin{proof}[Proof of Theorem~\ref{thm:convergence2}]
The first equality follows directly from
\eqref{eq:representationf} and \eqref{eq:projections}. For the
second sum, fix $x\in X$ and recall that there exists a unique
$Q=Q^m\in\mathscr{D}_m$ such that $x\in Q$. Thus,
\[
    \sum_{Q\in\mathscr{D}_m} \langle f,h^Q_0\rangle h^Q_0(x)
    = \langle f,h^{Q^m}_0\rangle h^{Q^m}_0(x)
    = \frac{1_{Q^m}(x)}{\mu(Q^m)} \int_{Q^m}f \, d\mu.
\]
If $\mu(X)=\infty$, then
\[
    \left\vert \frac{1_{Q^m}(x)}{\mu(Q^m)}
        \int_{Q^m} f \, d\mu \right\vert
    \leq \frac{\Norm{f}{L^p(\mu)}}{\mu(Q^m)^{1/p}}\to 0
    \text{ as } m\to-\infty
\]
by Lemma~\ref{lem:grows_to_X}. For the case $\mu(X)<\infty$,
the asserted convergence is clear.
\end{proof}

\section{Definitions and duality of dyadic $H^1_d(\protect\widetilde{X})$
and dyadic $\bmo_d(\protect\widetilde{X})$}
\label{sec:functionclasses} \setcounter{equation}{0}

We begin this section by recalling the definitions of
continuous product $H^1(\wX)$ and $\bmo(\wX)$ developed
in~\cite{HLW}. Here the underlying space $\wX = X_1\times
\cdots \times X_n$ is a product space of homogeneous type. Then
we introduce the definitions of the dyadic product Hardy and
$\bmo$ spaces, $H^1_{d,d}(\wX)$ and $\bmo_{d,d}(\wX)$,
associated to a system of dyadic cubes on~$\wX$. (For
simplicity of notation, we work in the setting of two
parameters, but everything goes through to the $n$-parameter
setting, including the material we quote from~\cite{HLW,
HLPW}.) These dyadic spaces are defined by means of the Haar
functions constructed in Section~\ref{sec:Haarfunctions} above.
We note that Tao used the same approach to define dyadic
product~$H^1$ and $\bmo$ on Euclidean underlying spaces, in the
note~\cite{Tao}. Finally, we prove the duality relation
\[
    (H^1_{d,d}(\wX))'
    = \bmo_{d,d}(\wX).
\]
Here our approach is to prove the duality of the product
sequence spaces $s_1$ and $c_1$ which are models for
$H^1_{d,d}(\wX)$ and $\bmo_{d,d}(\wX)$, and to pull this result
across by means of the lifting and projection operators~$T_L$
and~$T_P$.


We recall from~\cite{HLW} the definitions of the Hardy space $
H^1( \widetilde{X} )$ and the bounded mean oscillation space $
\bmo( \widetilde{X} )$.
These definitions rely on the orthonormal basis and the wavelet
expansion in $L^2(X)$ which were recently constructed by
Auscher and Hyt\"onen~\cite{AuHyt}. To state their result, we
must first recall the set $\{x_\alpha^k\}$ of {\it reference
dyadic points} as follows. Let $\delta$ be a fixed small
positive parameter (for example, as noted in Section~2.2 of
\cite{AuHyt}, it suffices to take $\delta\leq 10^{-3}
A_0^{-10}$). For $k = 0$, let $\mathscr{X}^0 :=
\{x_\alpha^0\}_\alpha$ be a maximal collection of 1-separated
points in~$ X $. Inductively, for $k\in\mathbb{Z}_+$, let
$\mathscr{X}^k := \{x_\alpha^k\} \supseteq \mathscr{X}^{k-1}$
and $\mathscr{X}^{-k} := \{x_\alpha^{-k}\} \subseteq
\mathscr{X}^{-(k-1)}$ be $\delta^k$- and
$\delta^{-k}$-separated collections in $\mathscr{X}^{k-1}$ and
$\mathscr{X}^{-(k-1)}$, respectively.

Lemma~2.1 in \cite{AuHyt} shows that, for all $k\in\mathbb{Z}$
and $x\in X$, the reference dyadic points satisfy
\begin{eqnarray}\label{delta sparse}
    \rho(x_\alpha^k,x_\beta^k)\geq\delta^k\ (\alpha\not=\beta),\hskip1cm
        \rho(x,\mathscr{X}^k)
    = \min_\alpha \, \rho(x,x_\alpha^k)
    < 2A_0\delta^k.
\end{eqnarray}
Also, taking $c_0 := 1$, $C_0 := 2A_0$ and $\delta \leq 10^{-3}
A_0^{-10}$, we see that $c_0$, $C_0$ and $\delta$ satisfy the
conditions in Theorem~2.2 of \cite{HK}. Therefore we may apply
Hyt\"onen and Kairema's construction (Theorem~2.2, \cite{HK}),
with the reference dyadic points $\{x^k_\alpha\}_{k\in\Z,
\alpha\in\mathscr{X}^k}$ playing the role of the points
$\{z^k_\alpha\}_{k\in\Z, \alpha\in\mathscr{A}_k}$, to conclude
that there exists a set of half-open dyadic cubes
\[
    \{Q_\alpha^k\}_{k\in\mathbb{Z},\alpha\in\mathscr{X}^k}
\]
associated with the reference dyadic points
$\{x_\alpha^k\}_{k\in\mathbb{Z},\alpha\in\mathscr{X}^k}$. We
call the reference dyadic point~$x_\alpha^k$ the \emph{center}
of the dyadic cube~$Q_\alpha^k$. We also identify with
$\mathscr{X}^k$ the set of indices~$\alpha$ corresponding to
$x_\alpha^k \in \mathscr{X}^{k}$.

Note that $\mathscr{X}^{k}\subseteq \mathscr{X}^{k+1}$ for
$k\in\mathbb{Z}$, so that every $x_\alpha^k$ is also a point of
the form $x_\beta^{k+1}$. We denote
$\mathscr{Y}^{k}:=\mathscr{X}^{k+1}\backslash \mathscr{X}^{k}$,
and relabel the points $\{x_\alpha^k\}_\alpha$ that belong
to~$\mathscr{Y}^k$ as $\{y_\alpha^k\}_\alpha$.

We now recall the orthonormal wavelet basis of $L^2(X)$
constructed by Auscher and Hyt\"onen.

\begin{thm}[\cite{AuHyt} Theorem 7.1]\label{theorem AH orth basis}
    Let $( X,\rho,\mu)$ be a space of homogeneous type with
    quasi-triangle constant $A_0$, and let
    \begin{equation}\label{eqn:defn_of_a}
        a
        := (1+2\log_2A_0)^{-1}.
    \end{equation}
    There exists an orthonormal wavelet basis
    $\{\psi_\alpha^k\}$, $k\in\mathbb{Z}$, $y_\alpha^k\in
    \mathscr{Y}^k$, of $L^2(X)$, having exponential decay
    \begin{eqnarray}\label{exponential decay}
        |\psi_\alpha^k(x)|
        \leq {C\over \sqrt{\mu(B(y_\alpha^k,\delta^k))}}
            \exp\Big(-\nu\Big( {\rho(y^k_\alpha,x)\over\delta^k}\Big)^a\Big),
    \end{eqnarray}
    H\"older regularity
    \begin{eqnarray}\label{Holder-regularity}
        |\psi_\alpha^k(x)-\psi_\alpha^k(y)|
        \leq {C\over \sqrt{\mu(B(y_\alpha^k,\delta^k))}}
            \Big( {\rho(x,y)\over\delta^k}\Big)^\eta
            \exp\Big(-\nu\Big( {\rho(y^k_\alpha,x)\over\delta^k}\Big)^a\Big)
    \end{eqnarray}
    for $\rho(x,y)\leq \delta^k$, and the cancellation property
    \begin{eqnarray}\label{cancellation}
        \int_X \psi_\alpha^k(x)\,d\mu(x) = 0,
        \qquad \text{for $k\in\mathbb{Z}$, $\ y_\alpha^k\in\mathscr{Y}^k$}.
    \end{eqnarray}
    Here $\delta$ is a fixed small parameter, say $\delta \leq
    10^{-3} A_0^{-10}$, and $C < \infty$, $\nu > 0$ and
    $\eta\in(0,1]$ are constants independent of $k$, $\alpha$,
    $x$ and~$y_\alpha^k$.
\end{thm}

Moreover, the wavelet expansion is given by
\begin{equation*}\label{eqn:AH_reproducing formula}
    f(x)
    = \sum_{k\in\mathbb{Z}}\sum_{\alpha \in \mathscr{Y}^k}
        \langle f,\psi_{\alpha}^k \rangle \psi_{\alpha}^k(x) \\
    = \sum_{k\in\mathbb{Z}}\Delta_kf(x),
\end{equation*}
in the sense of $L^2(X)$, where
\[
    \Delta_kf
    := \sum_{\alpha \in \mathscr{Y}^k}
        \langle f,\psi_{\alpha}^k \rangle \psi_{\alpha}^k(x)
\]
is the orthogonal projection onto the subspace $W_k$ spanned by
$\{\psi_{\alpha}^k\}_{\alpha \in \mathscr{Y}^k}$.




In what follows, we refer to the functions $\psi_\alpha^k$ as
\emph{wavelets}. Throughout Sections~\ref{sec:functionclasses}
and~\ref{sec:multiparameter} of this paper, $a$ denotes the
exponent from~\eqref{eqn:defn_of_a} and $\eta$ denotes the
H\"older-regularity exponent from~\eqref{Holder-regularity}.

We now consider the product setting $(X_1,\rho_1,\mu_1)\times
(X_2,\rho_2,\mu_2)$, where $(X_i,\rho_i,\mu_i)$, $i = 1$,~2, is
a space of homogeneous type as defined in Section~1. For $i =
1$, 2, let $A_0^{(i)}$ be the constant in the quasi-triangle
inequality~\eqref{eqn:quasitriangleineq}, and let $C_{\mu_i}$
be the doubling constant as in inequality~\eqref{doubling
condition}.
On each $X_i$, by Theorem~\ref{theorem AH orth basis}, there is
a wavelet basis~$\{\psi^{k_i}_{\alpha_i}\}$, with H\"older
exponent~$\eta_i$ as in inequality~\eqref{Holder-regularity}.


We refer the reader to~\cite{HLW}, Definitions~3.9 and~3.10 and
the surrounding discussion, for the definitions of the
space~$\GG$ of product test functions and its dual space
$(\GG)'$ of product distributions on the product space
$X_1\times X_2$. In \cite{HLW}, $\GG$ is denoted by
$\GG(\beta_1,\beta_2;\gamma_1,\gamma_2)$ and $(\GG)'$ is
denoted by $\GG(\beta_1,\beta_2;\gamma_1,\gamma_2)'$, where the
$\beta_i$ and $\gamma_i$ are parameters that quantify the size
and smoothness of the test functions, and $\beta_i \in
(0,\eta_i)$ where $\eta_i$ is the regularity exponent from
Theorem~\ref{theorem AH orth basis}. (In fact, in~\cite{HLW}
the theory is developed for $\beta_i \in (0,\eta_i]$, but for
simplicity here we only use $\beta_i \in (0,\eta_i)$ since that
is all we need.) We note that the one-parameter scaled
Auscher--Hyt\"onen wavelets $\psi_\alpha^k(x)/
\sqrt{\mu(B(y_\alpha^k,\delta^k))}$ are test functions, and
that their tensor products $\psi^{k_1}_{\alpha_1}(x)
\psi^{k_2}_{\alpha_2}(y) \big(\mu_1(B(y^{k_1}_{\alpha_1},
\delta_1^{k_1})) \mu_2(B(y^{k_2}_{\alpha_2},
\delta_2^{k_2}))\big)^{-1/2}$ are product test functions
in~$\GG$, for all $\beta_i \in (0,\eta_i]$ and all $\gamma_i >
0$, for $i = 1$,~2.
These facts follow from the theory in~\cite{HLW}, specifically
Definition~3.1 and the discussion after it, Theorem~3.3, and
Definitions~3.9 and 3.10 and the discussion between them.

As shown in Theorem~3.11 of~\cite{HLW}, the reproducing formula
\begin{equation}\label{product reproducing formula}
   f(x_1,x_2)
   = \sum_{k_1}\sum_{\alpha_1}\sum_{k_2}\sum_{\alpha_2}
    \langle f,\psi_{\alpha_1}^{k_1}\psi_{\alpha_2}^{k_2} \rangle
    \psi_{\alpha_1}^{k_1}(x_1)\psi_{\alpha_2}^{k_2}(x_2)
\end{equation}
holds in $\GG$ for all $\beta_i$, $\gamma_i \in
(0,\eta_i)$ for $i = 1$,~2, and hence by duality \eqref{product
reproducing formula} also holds in~$(\GG)'$.



Next we define the product Hardy space~$H^1$. This is the
special case $p = 1$ of the definition of $H^p$ in
Definition~5.1 in~\cite{HLW}.

\begin{defn}[\cite{HLW}, Definition~5.1]\label{def-H1}
    The \emph{product Hardy space~$H^1$ on $\widetilde{X}$} is
    defined as
    \[
        H^1(\widetilde{X})
        := \big\lbrace f \in (\GG)': S(f)\in L^1(\widetilde{X})\big\rbrace,
    \]
    where $S(f)$ is the Littlewood--Paley $g$-function related to the
    orthonormal basis introduced in~\cite{AuHyt} as follows:
    \begin{eqnarray}\label{g function}
        S(f)(x_1,x_2)
        :=\Big\{ \sum_{k_1,\alpha_1}\sum_{k_2,\alpha_2} \Big|
            \langle \psi_{\alpha_1}^{k_1}\psi_{\alpha_2}^{k_2},f \rangle
            \widetilde{\chi}_{Q_{\alpha_1}^{k_1}}(x_1)
            \widetilde{\chi}_{Q_{\alpha_2}^{k_2}}(x_2) \Big|^2 \Big\}^{1/2},
    \end{eqnarray}
    with $\widetilde{\chi}_{Q_{\alpha_1}^{k_1}}(x_1) :=
    \chi_{Q_{\alpha_1}^{k_1}}(x_1)\mu_1(Q_{\alpha_1}^{k_1})^{-1/2}$
    and $\widetilde{\chi}_{Q_{\alpha_2}^{k_2}}(x_2) :=
    \chi_{Q_{\alpha_2}^{k_2}}(x_2)\mu_2(Q_{\alpha_2}^{k_2})^{-1/2}$.
\end{defn}

As noted in~\cite{HLPW}, it follows from
Definition~\ref{def-H1} that $H^1(\wX)\cap L^2(\wX)$ is dense
in $H^1(\wX)$ with respect to the $H^1(\wX)$ norm~$\|f\|_{H^1}
:= \|Sf\|_{L^1(\wX)}$.


We first give the definition of atoms for the product Hardy
space. This definition is the special case with $q = 2$, $p =
1$ of Definition~4.7 in~\cite{HLPW}.

\begin{defn}[\cite{HLPW}, Definition~4.7]\label{def-of-H1-atom}
    A function $a(x_1,x_2)$ defined on $\widetilde{X}$ is called an
    \emph{atom of}~$H^1(\widetilde{X})$ if $a(x_1,x_2)$ satisfies the
    following properties (1)--(3):
    \begin{itemize}
    \item[(1)] supp $a\subset\Omega$, where $\Omega$ is an open
        set of $\widetilde{X}$ with finite measure;
    \item[(2)] $\|a\|_{L^2}\leq \mu(\Omega)^{-1/2}$; and
    \item[(3)] $a$ can be further decomposed into
        \emph{rectangle atoms}~$a_R$ associated to dyadic
        rectangles $R=Q_1\times Q_2$, satisfying the
        following properties (i)--(iii):

    \smallskip
    (i) supp $a_R\subset \overline{C}R$, where
    $\overline{C}$ is an absolute constant independent of
    $a$ and $R$;

    \smallskip
    (ii) $\int_{ X_1 }a_R(x_1,x_2)\,d\mu_1(x_1)=0$ for a.e.
    $x_2\in  X_2 $ and $\int_{ X_2 }a_R(x_1,x_2)\,d\mu_2(x_2)=0$

    \hskip.75cm for a.e. $x_1\in  X_1 $;

    \smallskip
    (iii) $a=\sum\limits_{R\in m(\Omega)}a_R$ and $
    \sum\limits_{R\in m(\Omega)}\|a_R\|_{L^2}^2 \leq
    \mu(\Omega)^{-1}$.
    \end{itemize}
\end{defn}



The special case $p = 1$, $q = 2$ of the following result
from~\cite{HLPW} shows that the atomic Hardy space defined in
terms of the atoms from Definition~\ref{def-of-H1-atom}
coincides with the dense subset $H^1(\widetilde{X}) \cap
L^2(\widetilde{X})$ of~$H^1(\widetilde{X})$. This atomic
decomposition will be used in Section~\ref{sec:multiparameter}.


\begin{thm}[\cite{HLPW}, Theorem~4.8]\label{theorem-Hp atom decomp}
    Suppose that $\max\big(\frac{ {\omega}_1}{ {\omega}_1 +
    \eta_1 }, \frac{ {\omega}_2}{ {\omega}_2+ \eta_2
    }\big)<p\leq 1<q<\infty$. Then $f\in L^q( X_1\times X_2
    )\cap H^p( X_1\times X_2 )$ if and only if $f$ has an
    atomic decomposition; that is,
    \begin{eqnarray}\label{atom decom}
        f=\sum_{i=-\infty}^\infty\lambda_ia_i,
    \end{eqnarray}
    where $a_i$ are $(p, q)$ atoms, $\sum_i|\lambda_i|^p<\infty,$
    and  the series converges in both $H^p( X_1\times X_2 )$ and
    $L^q( X_1\times X_2 )$. Moreover,
    \begin{eqnarray*}
        \|f\|_{H^p( X_1\times X_2 )}
        \sim \inf \Big\{ \Big[\sum_i|\lambda_i |^p \Big]^{1/p} :
            f = \sum_{i} \lambda_i a_i\Big\},
    \end{eqnarray*}
    where the infimum is taken over all decompositions as above and
    the implicit constants are independent of the $L^q( X_1\times
    X_2 )$ and $H^p( X_1\times X_2 )$ norms of $f.$
\end{thm}

We turn to the definition of the product $\bmo$ space.

\begin{defn}[\cite{HLW}, Definition~5.2]\label{def-BMO}
    We define the \emph{product $\bmo$ space on
    $\widetilde{X}$} as
    \[
        \bmo( \widetilde{X} )
        :=\big\{ f \in (\GG)': \mathcal{C}_1(f)\in L^\infty\},
    \]
    with $\mathcal{C}_1(f)$ defined as follows:
    \begin{eqnarray}\label{Carleson norm}
        \mathcal{C}_1(f)
        := \sup_{\Omega}\Big\{\frac{1}{\mu(\Omega)}
            \sum_{R=Q_{\alpha_1}^{k_1}\times Q_{\alpha_2}^{k_2}\subset \Omega}
            \big| \langle \psi_{\alpha_1}^{k_1}\psi_{\alpha_2}^{k_2},f \rangle
            \big|^2 \Big\}^{1/2},
    \end{eqnarray}
    where $\Omega$ runs over all open sets in $ \widetilde{X} $
    with finite measure, and $R$ runs over all dyadic rectangles
    contained in $\Omega$.
\end{defn}

We are now ready to define the dyadic versions of our function
spaces.
\begin{defn}\label{def-dyadic-H1}
    We define the \emph{dyadic product Hardy space $H^1$ on
    $\widetilde{X}$} as
    \[
        H_{d,d}^1(\widetilde{X})
        := \big\lbrace f \in L^1(X_1\times X_2):
            S_{d,d}(f)\in L^1(\widetilde{X})\big\rbrace,
    \]
    where $S_{d,d}(f)$ is the Littlewood-Paley $g$-function related
    to the Haar orthonormal basis built in the previous sections,
    as follows:
    \begin{align}\label{dyadic-g-function}
        S_{d,d}(f)(x_1,x_2)
        := \Big\{\sum_{Q_1\in\mathcal{D}_1} \sum_{Q_2\in\mathcal{D}_2}
            \sum_{u_1=1}^{M_{Q_1}-1}\sum_{u_2=1}^{M_{Q_2}-1}
            \big| \langle f, h_{Q_1}^{u_1}h_{Q_2}^{u_2}\rangle
            \widetilde{\chi}_{Q_1}(x_1)\widetilde{\chi}_{Q_2}(x_2)\big|^2 \Big\}^{1/2},
    \end{align}
    where $\widetilde{\chi}_{Q_1}(x_1) :=
    \chi_{Q_1}(x_1)\mu_1(Q_1)^{-1/2}$
    and $\widetilde{\chi}_{Q_2}(x_2) :=
    \chi_{Q_2}(x_2)\mu_2(Q_2)^{-1/2}$.
\end{defn}

\begin{defn}\label{def-dyadic-BMO}
    We define the \emph{dyadic product $\bmo$ space on $\widetilde{X}$} as
    \[
        \bmo_{d,d}( \widetilde{X} )
        := \big\{ f \in L^1_{\text{loc}} : \mathcal{C}_{d,d}(f)\in L^\infty\},
    \]
    with $\mathcal{C}_{d,d}(f)$ defined as follows:
    \begin{align}\label{dyadic-BMO norm}
        \mathcal{C}_{d,d}(f)
        := \sup_{\Omega}\bigg\{\frac{1}{\mu(\Omega)}
            \sum_{\substack{R=Q_{1}\times Q_{2}\in\mathcal{D}_1\times\mathcal{D}_2,\\ R\subset \Omega}}
            \sum_{u_1=1}^{M_{Q_1}-1}\sum_{u_2=1}^{M_{Q_2}-1}
            \big| \langle f, h_{Q_1}^{u_1}h_{Q_2}^{u_2}\rangle \big|^2
            \bigg\}^{1/2},
    \end{align}
    where $\Omega$ runs over all open sets in $ \widetilde{X} $
    with finite measure, and $R$ runs over all dyadic rectangles
    contained in~$\Omega$.
\end{defn}

\begin{remark}
    We note that our definitions of $H^1_{d,d}(\wX)$ and
    $\bmo_{d,d}(\wX)$ can be rephrased in terms of martingales,
    as follows. However, in the rest of this paper we work with
    the definitions above in terms of Haar functions.

    For $H^1_{d,d}(\wX)$: The dyadic square function $S_{d,d}(f)$ used
    in Definition~\ref{def-dyadic-H1} can be
    replaced by 
    \begin{eqnarray}\label{dyadic-g-function -via matingale}
        \widehat{S}_{d,d}(f)(x_1,x_2)
        := \Big\{ \sum_{k_1,\alpha_1}\sum_{k_2,\alpha_2}
            \Big|\Db_{Q_{\alpha_1}^{k_1}}\Db_{Q_{\alpha_2}^{k_2}}(f)(x_1,x_2)\Big|^2
            \Big\}^{1/2},
    \end{eqnarray}
    where $\Db_{Q_{\alpha_1}^{k_1}}
    \Db_{Q_{\alpha_2}^{k_2}}(f)(x_1,x_2)$ means we first apply
    the orthogonal projection $\Db_{Q_{\alpha_2}^{k_2}}$ to
    $f(x_1,\cdot )$ for a.e.~$x_1$, and then apply
    $\Db_{Q_{\alpha_1}^{k_1}}$ to the resulting function of~$x_1$.

    The space
    \[
        \widehat{H}_{d,d}^1(\widetilde{X})
        := \big\lbrace f \in L^1(X_1\times X_2):
            \widehat{S}_{d,d}(f)\in L^1(\widetilde{X})\big\rbrace
    \]
    coincides with $H_{d,d}^1(\widetilde{X})$, and moreover
    $\|S_{d,d}f\|_{L^1(X_1\times X_2)} \sim
    \|\widehat{S}_{d,d}f\|_{L^1(X_1\times X_2)}$; see
    \cite[Section~3.2.3]{Tre} and~\cite{FJ}. The advantage of
    using the square function $S_{d,d}$ is that only the
    absolute value of the individual coefficients appears in
    the definition, and there is no further cancellation
    involved. Also notice that using the local martingale
    differences and expectations we can write $S_{d,d}$ as
    follows:
    \begin{eqnarray}\label{eqn:Square function -via martingale}
        {S}_{d,d}(f)(x_1,x_2)
        = \Big\{ \sum_{k_1,\alpha_1}\sum_{k_2,\alpha_2}
            \E_{Q_{\alpha_1}^{k_1}}\E_{Q_{\alpha_2}^{k_2}}
            \Big|\Db_{Q_{\alpha_1}^{k_1}}\Db_{Q_{\alpha_2}^{k_2}}(f)(x_1,x_2)\Big|^2
            \Big\}^{1/2}.
    \end{eqnarray}


    For $\bmo_{d,d}(\wX)$: The quantity $\mathcal{C}_{d,d}(f)$
    used in Definition~\ref{def-dyadic-BMO} can be written as
    follows:
    \begin{align}\label{dyadic-BMO norm via martingale}
       \mathcal{C}^{d,d}_1(f)
       = \sup_{\Omega}\Big\{{1\over\mu(\Omega)}
            \sum_{\substack{R = Q_{\alpha_1}^{k_1}\times
            Q_{\alpha_2}^{k_2}\in\mathcal{D}_1\times\mathcal{D}_2,\\ R\subset \Omega}}
            \int_{X_1\times X_2}\big|
            \Db_{Q_{\alpha_1}^{k_1}}\Db_{Q_{\alpha_2}^{k_2}}(f)(x_1,x_2)
            \big|^2d\mu_1(x_1)d\mu_2(x_2) \Big\}^{1/2}.
    \end{align}
    where $\Omega$ runs over all open sets in $ \widetilde{X} $
    with finite measure, and $R$ runs over all dyadic
    rectangles contained in $\Omega$. This time we get equality
    because we are integrating the square of the local
    martingale differences (orthogonal projection onto the
    subspace generated by the Haar functions associated to the
    rectangle $R=Q_{\alpha_1}^{k_1}\times Q_{\alpha_2}^{k_2}$)
    and we can use Plancherel to replace the square of the
    $L^2$-norm by the sum of the squares of the Haar
    coefficients.

    For each dyadic cube~$Q$, the space
    \[
        W(Q)
        := \lspan\{f(x) \colon \supp f\subset Q, \text{$f|_R$ is constant
            for each}~R\in \text{ch}(Q),~\text{and}~\smallint f = 0\}
    \]
    of all functions supported on~$Q$ that are constant on each
    child of~$Q$ and have mean zero is a finite-dimensional
    vector space, with dimension $\dim(W(Q)) = M_Q - 1$. It is
    a subspace of the space $V(Q)$ defined in the proof of
    Theorem~\ref{thm:HaarFuncProp2} above. There are many
    possible bases of Haar functions for $W(Q)$, in addition to
    the basis we chose above. However, it follows from the
    remarks in the three preceding paragraphs that the
    definitions of $H^1_{d,d}(\widetilde{X})$ and
    $\bmo_{d,d}(\widetilde{X})$ are independent of the Haar
    functions chosen on each finite-dimensional subspace
    $W(Q_{\alpha_i}^{k_i})$, since in the definitions we are using the
    orthogonal projections $\Db_{Q_{\alpha_1}^{k_1}}
    \Db_{Q_{\alpha_2}^{k_2}}(f)$.
%
%
\end{remark}

\m

We now establish the duality $\big(H^1_{d,d}(\wX)\big)' =
\bmo_{d,d}(\wX)$. Following the approach used in~\cite{HLL} for
the continuous case, we do so by passing to the corresponding
duality result for certain sequence spaces $s^1$ and~$c^1$,
which are models for $H^1_{d,d}(\wX)$ and $\bmo_{d,d}(\wX)$
respectively. These sequence spaces are part of the homogeneous
Triebel--Lizorkin family of sequence spaces adapted to the
product setting. In the notation used in \cite{FJ}
and~\cite{Tre}, $s^1 = \dot{f}_1^{0,2}$ and $c^1 =
\dot{f}^{0,2}_{\infty}$.

We begin with the definitions and properties of these sequence
spaces.

We define $s^1$ to be the sequence space consisting of the
sequences $s = \{s_{Q_1,u_1,Q_2,u_2}\}$ of complex numbers with
\begin{align}\label{e4.1}
    \|s\|_{s^1}
    :=\Big\|\Big\{\sum_{Q_1\in\mathcal{D}_1} \sum_{Q_2\in\mathcal{D}_2}
        \sum_{u_1=1}^{M_{Q_1}-1}\sum_{u_2=1}^{M_{Q_2}-1}
        \big| s_{Q_1,u_1,Q_2,u_2} \widetilde{\chi}_{Q_1}(x_1)
        \widetilde{\chi}_{Q_2}(x_2)\big|^2
        \Big\}^{1/2}\Big\|_{L^1(X_1\times X_2)}
    < \infty.
\end{align}

We define $c^1$ to be the sequence space consisting of the
sequences $t = \{t_{Q_1,u_1,Q_2,u_2}\}$ of complex numbers with
\begin{eqnarray}\label{e4.2}
    \|t\|_{c^1}
    :=\sup_{\Omega}\Big(\omegareverse
        \sum_{R=Q_{1}\times Q_{2}\in\mathcal{D}_1\times\mathcal{D}_2,\ R\subset \Omega}
        \sum_{u_1=1}^{M_{Q_1}-1}\sum_{u_2=1}^{M_{Q_2}-1}
        |t_{Q_1,u_1,Q_2,u_2}|^2
        \Big)^{1/2}<\infty,
\end{eqnarray}
where the supremum is taken over all open sets~$\Omega$
in~$\wX$ with finite measure.

The main result for these sequence spaces is their duality.

\begin{thm}\label{prop-dual-of-sequence-space}
    The following duality result holds: $\big(s^1\big)^{'} = c^1$.
\end{thm}

\begin{proof}
First, we prove that for each $t\in c^1$, if
\begin{align}\label{innner product of the sequence}
    L(s)
    =\langle s,t\rangle:=\sum_{Q_1\in\mathcal{D}_1} \sum_{Q_2\in\mathcal{D}_2}
        \sum_{u_1=1}^{M_{Q_1}-1} \sum_{u_2=1}^{M_{Q_2}-1}
        s_{Q_1,u_1,Q_2,u_2}\cdot {\overline t_{Q_1,u_1,Q_2,u_2}},\
        \textup{for all } s \in s^1,
\end{align}
then
\begin{eqnarray}\label{L(s) bounded by s norm and t norm}
|L(s)|\leq C \|s\|_{s^1}\|t\|_{c^1}.
\end{eqnarray}

To see this, define for each $k\in\Z$ the following subsets
$\Omega_k$ and $\widetilde{\Omega}_k$ of $X_1\times X_2$ and
the subcollection~$B_k$ of certain rectangles in
$\mathcal{D}_1\times \mathcal{D}_2$:
\begin{eqnarray*}
    \Omega_k
    :=\Bigg\{(x_1,x_2)\in X_1\times X_2\ :
        \Big\{\sum_{Q_1\in\mathcal{D}_1} \sum_{Q_2\in\mathcal{D}_2}
        \sum_{u_1=1}^{M_{Q_1}-1}\sum_{u_2=1}^{M_{Q_2}-1}
        \big| s_{Q_1,u_1,Q_2,u_2}
        \widetilde{\chi}_{Q_1}(x_1)\widetilde{\chi}_{Q_2}(x_2)\big|^2
        \Big\}^{1/2}>2^k
        \Bigg\},
\end{eqnarray*}
\begin{eqnarray*}
    B_k
    :=\Big\{R=Q_1\times Q_2 \in \mathcal{D}_1\times \mathcal{D}_2:\
        \mu(\Omega_k\cap R)>{1\over 2}\mu (R),\
        \mu(\Omega_{k+1}\cap R)
    \leq\frac{\displaystyle1}{\displaystyle2}\mu(R)
    \Big\},\
\end{eqnarray*}
and
\begin{eqnarray*}
    \widetilde{\Omega}_k
    :=\Big\{(x_1,x_2)\in X_1\times X_2\ : M_s(\chi_{\Omega_k})>
    \frac{\displaystyle1}{\displaystyle2}\Big\},
\end{eqnarray*}
where $M_s$ is the strong maximal function on $ X_1\times X_2
$.

\begin{remark}\label{remark:duality seq spaces}
    First notice that since $s\in s^1$ we have $|\Omega_k | <
    \infty$. Second, each rectangle $R$ belongs to exactly one
    set $B_k$. Third, note that $|\widetilde{\Omega}_k| \leq
    C|\Omega_k|$ because the strong maximal function is bounded
    on~$L^2(\wX)$. Fourth, if $R\in B_k$ then $R\subset
    \widetilde{\Omega}_k$, therefore $\cup_{R\in B_k}R \subset
    \widetilde{\Omega}_k$ and $B_k\subset
    \{R : R\subset\widetilde{\Omega}_k\}$. Fifth,  if $R\in B_k$ then
    $\mu( (\widetilde{\Omega}_k\backslash \Omega_{k+1})\cap R)\geq
    \mu(R)/2$.
\end{remark}

By H\"{o}lder's inequality, the linear functional~$L(s)$
from~(\ref{innner product of the sequence}) satisfies
\begin{eqnarray}\label{estimate-of-Ls}
    |L(s)|
    &\leq& \sum_k \bigg(\sum_{R=Q_{1}\times Q_{2}\in\mathcal{D}_1\times\mathcal{D}_2,\ R\in B_k}
        \sum_{u_1=1}^{M_{Q_1}-1}\sum_{u_2=1}^{M_{Q_2}-1}
        |s_{Q_1,u_1,Q_2,u_2}|^2\bigg)^{1\over 2} \nonumber\\
    &&\times \bigg(\sum_{R=Q_{1}\times Q_{2}\in\mathcal{D}_1\times\mathcal{D}_2,\ R\in B_k}
        \sum_{u_1=1}^{M_{Q_1}-1}\sum_{u_2=1}^{M_{Q_2}-1}
        |t_{Q_1,u_1,Q_2,u_2}|^2\bigg)^{1\over 2}  \nonumber\\
    &\leq & \sum_k \mu (\widetilde{\Omega}_k)^{1\over 2}
        \bigg(\sum_{R=Q_{1}\times Q_{2}\in\mathcal{D}_1\times\mathcal{D}_2,\ R\in B_k}
        \sum_{u_1=1}^{M_{Q_1}-1}\sum_{u_2=1}^{M_{Q_2}-1}
        |s_{Q_1,u_1,Q_2,u_2}|^2\bigg)^{1\over 2}\\
    &&\times \bigg(\frac{\displaystyle1}{\displaystyle\mu(\widetilde{\Omega}_k)}
    \sum_{R=Q_{1}\times Q_{2}\in\mathcal{D}_1\times\mathcal{D}_2,\ R\in B_k}
        \sum_{u_1=1}^{M_{Q_1}-1}\sum_{u_2=1}^{M_{Q_2}-1}
        |t_{Q_1,u_1,Q_2,u_2}|^2\bigg)^{1\over 2}   \nonumber\\
    &\leq & \sum_k \mu (\widetilde{\Omega}_k)^{1\over 2}
        \bigg(\sum_{R=Q_{1}\times Q_{2}\in\mathcal{D}_1\times\mathcal{D}_2,\ R\in B_k}
        \sum_{u_1=1}^{M_{Q_1}-1}\sum_{u_2=1}^{M_{Q_2}-1}
        |s_{Q_1,u_1,Q_2,u_2}|^2\bigg)^{1\over 2}   \|t\|_{c^1}. \nonumber
\end{eqnarray}
In the last inequality we have used the fact that $B_k\subset
\{R:R\subset\widetilde{\Omega}_k\}$, as stated in
Remark~\ref{remark:duality seq spaces}.

Next we point out that from the definition of $\Omega_{k + 1}$
and Remark~\ref{remark:duality seq spaces}, we have
\begin{eqnarray*}
    &&\int_{\widetilde{\Omega}_k \backslash \Omega_{k+1}}
        \sum_{R=Q_{1}\times Q_{2}\in\mathcal{D}_1\times\mathcal{D}_2,\ R\in B_k}
        \sum_{u_1=1}^{M_{Q_1}-1}\sum_{u_2=1}^{M_{Q_2}-1}
        |s_{Q_1,u_1,Q_2,u_2}|^2
        {\chi_{Q_1}(x_1)\over\mu_1(Q_1)}{\chi_{Q_2}(x_2)\over\mu_2(Q_2)} d\mu_1(x_1)d\mu_2(x_2)\\
     &&\leq 2^{2(k+1)} \mu (\widetilde{\Omega}_k \backslash\Omega_{k+1})
     \leq C 2^{2k}\mu (\Omega_k).
\end{eqnarray*}
Moreover, by the Monotone Convergence Theorem and
Remark~\ref{remark:duality seq spaces}, we have
\begin{eqnarray*}
    &&\int_{\widetilde{\Omega}_k \backslash \Omega_{k+1}}
        \sum_{R=Q_{1}\times Q_{2}\in\mathcal{D}_1\times\mathcal{D}_2,\ R\in B_k}
        \sum_{u_1=1}^{M_{Q_1}-1}\sum_{u_2=1}^{M_{Q_2}-1}
        |s_{Q_1,u_1,Q_2,u_2}|^2 {\chi_{Q_1}(x_1)\over\mu_1(Q_1)}
        {\chi_{Q_2}(x_2)\over\mu_2(Q_2)} d\mu_1(x_1)d\mu_2(x_2) \\
    &&\geq \sum_{R=Q_{1}\times Q_{2}\in\mathcal{D}_1\times\mathcal{D}_2,\ R\in B_k}
        \sum_{u_1=1}^{M_{Q_1}-1}\sum_{u_2=1}^{M_{Q_2}-1}
        |s_{Q_1,u_1,Q_2,u_2}|^2  {1\over\mu_1(Q_1)}{1\over\mu_2(Q_2)}
        \mu(\widetilde{\Omega}_k \backslash \Omega_{k+1}\cap R\big) \\
    &&\geq \frac{\displaystyle1}{\displaystyle2}
        \sum_{R=Q_{1}\times Q_{2}\in\mathcal{D}_1\times\mathcal{D}_2,\ R\in B_k}
        \sum_{u_1=1}^{M_{Q_1}-1}\sum_{u_2=1}^{M_{Q_2}-1}
        |s_{Q_1,u_1,Q_2,u_2}|^2.
\end{eqnarray*}
As a consequence, we obtain that
\[
    \sum_{R=Q_{1}\times Q_{2}\in\mathcal{D}_1\times\mathcal{D}_2,\ R\in B_k}
        \sum_{u_1=1}^{M_{Q_1}-1}\sum_{u_2=1}^{M_{Q_2}-1}
        |s_{Q_1,u_1,Q_2,u_2}|^2
    \leq C 2^{2k}\mu (\Omega_k).
\]
Substituting this back into the last inequality of
(\ref{estimate-of-Ls}) yields that
\begin{eqnarray*}
    |L(s)|
    \leq  C\sum_k \mu (\widetilde{\Omega}_k)^{1\over 2}
        2^k\mu (\Omega_k)^{1\over2}
        \|t\|_{c^1}\leq C\sum_k   2^k \mu (\Omega_k)\|t\|_{c^1}
    \leq C\|s\|_{s^1}\|t\|_{c^1}.
\end{eqnarray*}

Conversely, we need to verify that for any $L\in
\big(s^1\big)^{'}$, there exists $t\in c^1$ with
$\|t\|_{c^1}\leq \|L\|$ such that for all $s\in s^1$,
$L(s)=\sum\limits_{Q_1,u_1,Q_2,u_2}
s_{Q_1,u_1,Q_2,u_2}\overline{t_{Q_1,u_1,Q_2,u_2}}$.

Now let $s^{Q_1,u_1,Q_2,u_2}$ be the sequence that equals 0
everywhere except at the $(Q_1,u_1,Q_2,u_2)$ entry where it
equals 1. Then it is easy to see that
$\|s^{Q_1,u_1,Q_2,u_2}\|_{s^1} = 1$. For all $s\in s^1$, we
have $s = \sum\limits_{Q_1,u_1,Q_2,u_2} s_{Q_1,u_1,Q_2,u_2}
s^{Q_1,u_1,Q_2,u_2}$, where the limit holds in the norm of
$s^1$. For each $L\in \big(s^1\big)^{'}$, let
$\overline{t_{Q_1,u_1,Q_2,u_2}} = L(s^{Q_1,u_1,Q_2,u_2})$.
Since $L$ is a bounded linear functional on $s^1$, we see that
\[
    L(s)
    = L\Big(\sum\limits_{Q_1,u_1,Q_2,u_2} s_{Q_1,u_1,Q_2,u_2}
        s^{Q_1,u_1,Q_2,u_2}\Big)
    = \sum\limits_{Q_1,u_1,Q_2,u_2} s_{Q_1,u_1,Q_2,u_2}
        \overline{t_{Q_1,u_1,Q_2,u_2}}.
\]
Let $t = \{t_{Q_1,u_1,Q_2,u_2}\}$. Then we only need to check
that $\|t\|_{c^1} \leq \|L\|$.


For each fixed open set $\Omega\subset X_1\times X_2 $ with
finite measure, let $\bar{\mu}$ be a new measure such that
$\bar{\mu}(R)=\frac{\displaystyle \mu(R)}{\displaystyle
\mu(\Omega)}$ when $ R\subset\Omega$, and $\bar{\mu}(R)=0$ when
$ R\nsubseteq\Omega$. And let $l^2(\bar{\mu})$ be a sequence
space such that when $s\in l^2(\bar{\mu})$,
\[
    \bigg(\sum_{R=Q_{1}\times Q_{2}\in\mathcal{D}_1\times\mathcal{D}_2,\ R\subset \Omega}
        \sum_{u_1=1}^{M_{Q_1}-1} \sum_{u_2=1}^{M_{Q_2}-1}
        |s_{Q_1,u_1,Q_2,u_2}|^2
        \frac{\displaystyle \mu(R)}{\displaystyle \mu(\Omega)}\bigg)^{1/2}
    < \infty.
\]
It is easy to see that $(l^2(\bar{\mu}))' = l^2(\bar{\mu})$.
Next,
\begin{align}\label{tQ}
    &\Big(\omegareverse \sum_{R=Q_{1}\times Q_{2}\in\mathcal{D}_1\times\mathcal{D}_2,\ R\subset \Omega}
        \sum_{u_1=1}^{M_{Q_1}-1}\sum_{u_2=1}^{M_{Q_2}-1}
        |t_{Q_1,u_1,Q_2,u_2}|^2 \Big)^{1/2}\\
    &=\Big\| \big\{\mu(R)^{-1/2}t_{Q_1,u_1,Q_2,u_2}\big\}\Big\|_{l^2(\bar{\mu})}\nonumber\\
    &=\sup_{s:\ \|s\|_{l^2(\bar{\mu})}\leq 1}
        \bigg|\sum_{R=Q_{1}\times Q_{2}\in\mathcal{D}_1\times\mathcal{D}_2,\ R\subset \Omega}
        \sum_{u_1=1}^{M_{Q_1}-1}\sum_{u_2=1}^{M_{Q_2}-1}
        \mu(R)^{-1/2}t_{Q_1,u_1,Q_2,u_2} \cdot s_{Q_1,u_1,Q_2,u_2} \cdot
        \frac{\displaystyle \mu(R)}{\displaystyle\mu(\Omega)}\bigg|\nonumber\\
    &=\sup_{s:\ \|s\|_{l^2(\bar{\mu})}\leq 1}
        \bigg|L\bigg(\Big\{s_{Q_1,u_1,Q_2,u_2}\,
        {\mu(R)^{1/2}\over\mu(\Omega)}\Big\}_{R=Q_{1}
        \times Q_{2}\in\mathcal{D}_1\times\mathcal{D}_2,\ R\subset \Omega,
        u_1=1,\ldots, M_{Q_1-1}, u_2=1,\ldots, M_{Q_2-1}}\bigg)\bigg|\nonumber\\
    &\leq\sup_{s:\ \|s\|_{l^2(\bar{\mu})}\leq 1}
        \big\|L\big\|\cdot\bigg\|\Big\{s_{Q_1,u_1,Q_2,u_2}\,
        {\mu(R)^{1/2}\over\mu(\Omega)}\Big\}_{R=Q_{1}\times Q_{2}\in\mathcal{D}_1\times\mathcal{D}_2,\
        R\subset \Omega, u_1=1,\ldots, M_{Q_1-1}, u_2=1,\ldots, M_{Q_2-1}}\bigg\|_{s^1}.\nonumber
\end{align}
By \eqref{e4.1} and the Cauchy--Schwarz inequality, we have
\begin{eqnarray*}
    &&\bigg\| \Big\{s_{Q_1,u_1,Q_2,u_2}\,
        {\mu(R)^{1/2}\over\mu(\Omega)}\Big\}_{R=Q_{1}\times Q_{2}\in\mathcal{D}_1\times\mathcal{D}_2,\
        R\subset \Omega, u_1=1,\ldots, M_{Q_1-1}, u_2=1,\ldots, M_{Q_2-1}}  \bigg\|_{s^1}\\
    && \leq \bigg(\sum_{R=Q_{1}\times Q_{2}\in\mathcal{D}_1\times\mathcal{D}_2,\ R\subset \Omega}
        \sum_{u_1=1}^{M_{Q_1}-1}\sum_{u_2=1}^{M_{Q_2}-1}
        |s_{Q_1,u_1,Q_2,u_2}|^2\frac{\displaystyle \mu(R)}{\displaystyle \mu(\Omega)}\bigg)^{1/2}.
\end{eqnarray*}
Hence, we obtain that (\ref{tQ}) is bounded by
\begin{eqnarray*}
    \sup_{s:\ \|s\|_{l^2(\bar{\mu})} \leq 1}
    \|L\|\cdot\|s\|_{l^2(\bar{\mu})} \leq \|L\|,
\end{eqnarray*}
which implies that
\[
    \|t\|_{c^1}\leq \|L\|.
    \qedhere
\]
\end{proof}

We now define $s^2$ (corresponding to $\dot{f}_2^{0,2}\sim
\ell^2$) to be the sequence space consisting of the sequences
$s = \{s_{Q_1,u_1,Q_2,u_2}\}$ of complex numbers with
\begin{align}\label{e4.3}
    \|s\|_{s^2}
    :=\Big\|\Big\{\sum_{Q_1\in\mathcal{D}_1} \sum_{Q_2\in\mathcal{D}_2}
        \sum_{u_1=1}^{M_{Q_1}-1}\sum_{u_2=1}^{M_{Q_2}-1}
        \big| s_{Q_1,u_1,Q_2,u_2} \widetilde{\chi}_{Q_1}(x_1)
        \widetilde{\chi}_{Q_2}(x_2)\big|^2
        \Big\}^{1/2}\Big\|_{L^2(X_1\times X_2)}<\infty.
\end{align}

Then we have the following density argument.

\begin{prop}\label{prop density-of-sequence-space}
    The set $s^1\cap s^2$ is dense in $s^1$ in terms of the
    $s^1$ norm. Moreover, $c^1\cap s^2$ is dense in $c^1$ in
    terms of the weak type convergence as follows: for each
    $t\in c^1$, there exists $\{t_n\}_n\subset c^1\cap s^2$
    such that $\langle s,t_n \rangle\rightarrow \langle
    s,t\rangle$ when $n$ tends to $\infty$, for all $s\in s^1$,
    where $\langle s,t\rangle$ is the inner product defined as
    in~\eqref{innner product of the sequence}.
\end{prop}

\begin{proof}
For each sequence $s=\{s_{Q_1,u_1,Q_2,u_2}\}\in s^1$, we define
a truncated sequence $s_n$ as follows:
\begin{equation}\label{s n}
    s_n
    := \big\{ s^n_{Q_1,u_1,Q_2,u_2} \big\},
    \quad
    s^n_{Q_1,u_1,Q_2,u_2}
    := s_{Q_1,u_1,Q_2,u_2}\chi_{A_n}(Q_1,Q_2),
\end{equation}
where $A_n=\{(Q_1,Q_2): k_1,k_2\in[-n,n], Q_1 \in
\mathscr{D}_{k_1},  Q_1\subset B(x_1^0,n), Q_2 \in
\mathscr{D}_{k_2}, Q_2\subset B(x_2^0,n)  \}$ for each positive
integer $n$, where $x_1^0$ and $x_2^0$ are arbitrary fixed
points in $X_1$ and $X_2$, respectively.



Then $s_n$ is in $s^1$ with $\|s_n\|_{s^1}\leq \|s\|_{s^1}$. We
also have that $s\in s^2$ with
\[
    \|s_n\|_{s^2}
    = \Big(\sum_{ k_1,k_2\in[-n,n],
        Q_1 \in \mathscr{D}_{k_1}, Q_1\subset B(x_1^0,n),
        Q_2 \in \mathscr{D}_{k_2}, Q_2\subset B(x_2^0,n) }
        |s_{Q_1,u_1,Q_2,u_2}|^2\Big)^{1\over 2}
    < \infty.
\]
Moreover, it is easy to check that $s_n$ tends to $s$ in the
sense of $s^1$ norm. As a consequence, we have that $s^1\cap
s^2$ is dense in $s^1$ in terms of the $s^1$ norm.

\medskip
Now we turn to $c^1$. For each $t = \{t_{Q_1,u_1,Q_2,u_2}\}\in
c^1$, we define $t_n$ analogously to $s_n$. Then, similarly, we
have $t_n\in c^1\cap s^2$ with $\|t_n\|_{c^1}\leq \|t\|_{c^1}$
for each positive integer~$n$. Moreover, since $|\langle
s,t_n\rangle|$ and $|\langle s,t\rangle|$ are both bounded by $
C\|s\|_{s^1}\|t\|_{c^1}$ for all $s\in s^1$, we have
\begin{eqnarray*}
    \langle s,t_n\rangle
    &=& \sum_{ k_1,k_2\in[-n,n],
        Q_1 \in \mathscr{D}_{k_1}, Q_1\subset B(x_1^0,n),
        Q_2 \in \mathscr{D}_{k_2}, Q_2\subset B(x_2^0,n)  }
        \sum_{u_1=1}^{M_{Q_1}-1}\sum_{u_2=1}^{M_{Q_2}-1}
        s_{Q_1,u_1,Q_2,u_2}\,t_{Q_1,u_1,Q_2,u_2}\\
    &\rightarrow& \sum_{Q_1\in\mathcal{D}_1} \sum_{Q_2\in\mathcal{D}_2}
        \sum_{u_1=1}^{M_{Q_1}-1}\sum_{u_2=1}^{M_{Q_2}-1} s_Q\,t_Q
        = \langle s,t\rangle
\end{eqnarray*}
as $n\rightarrow +\infty$, which shows that $t_n$ tends to $t$
in the weak type convergence.
\end{proof}

Now we define the lifting and projection operators as follows.

\begin{defn}\label{def-of-lifting-operator-on-product-case}
    For functions $f\in L^2(X_1\times X_2)$, define the lifting
    operator $T_{L}$ by
    \begin{eqnarray}\label{lifting-operator}
        T_L(f)
        = \big\{ \langle f, h_{Q_1}^{u_1}h_{Q_2}^{u_2} \rangle
            \big\}_{Q_1,u_1,Q_2,u_2}.
    \end{eqnarray}
\end{defn}

\begin{defn}\label{def-of-projection-operator-on-product-case}
    For sequences $s = \{s_{Q_1,u_1,Q_2,u_2}\}$, define the
    projection operator $T_{P}$ by
    \begin{eqnarray}\label{projection-operator}
        T_P(s)
        = \sum_{Q_1\in\mathcal{D}_1} \sum_{Q_2\in\mathcal{D}_2}
            \sum_{u_1=1}^{M_{Q_1}-1} \sum_{u_2=1}^{M_{Q_2}-1}
            s_{Q_1,u_1,Q_2,u_2} h_{Q_1}^{u_1}(x_1)h_{Q_2}^{u_2}(x_2).
    \end{eqnarray}
\end{defn}

Then it is clear that
\[
    f
    = T_P\circ T_L(f)
\]
in the sense of $L^2(X_1\times X_2)$.

Next we give the following two auxiliary propositions, which
show that the lifting operator~$T_L$ maps $ H_{d,d}^1(X_1\times
X_2)\cap L^2(X_1\times X_2)$ to $s^1\cap s^2$, and maps $
\bmo_{d,d}(X_1\times X_2) \cap L^2(X_1\times X_2)$ to $c^1\cap
s^2$.

\begin{prop}\label{prop-of-Hp-to-sp}
    For all $f\in L^2(X_1\times X_2)\cap H^1_{d,d}(X_1\times
    X_2)$, we have
    \begin{eqnarray}\label{Hp to sp}
        \|T_L(f)\|_{s^1}
        \lesssim \|f\|_{ H^1_{d,d}(X_1\times X_2)}.
    \end{eqnarray}
\end{prop}

\begin{prop}\label{prop-of-CMOp-to-cp}
    For all $f\in \bmo_{d,d}( X_1\times X_2)$, we have
    \begin{eqnarray}\label{CMOp to cp}
        \|T_L(f)\|_{c^1}
        \lesssim \mathcal{C}_1^d(f).
    \end{eqnarray}
\end{prop}

\begin{prop}\label{prop-of-sp-to-Hp}
    For all $s\in s^1\cap s^2$, we have
    \begin{eqnarray}\label{sp to Hp}
        \|T_P(s)\|_{ H^1_{d,d}(X_1\times X_2)}
        \lesssim \|s\|_{s^1}.
    \end{eqnarray}
\end{prop}

\begin{prop}\label{prop-of-cp-to-CMOp}
    For all $t\in c^1 \cap s^2$, we have
    \begin{eqnarray}\label{cp to CMOp}
        \mathcal{C}_1^d(T_P(t))
        \lesssim \|t\|_{c^1} .
    \end{eqnarray}
\end{prop}

These four propositions follow directly from the definitions of
the Hardy spaces $H^1_{d,d}(X_1\times X_2)$ and the sequence
space~$s^1$, and from the definitions of the $\bmo$ space
$\bmo_{d,d}(X_1\times X_2)$ and the sequence space $c^1$.

\begin{thm}\label{thm:duality_dyadic_H1_BMO}
    For each $\varphi\in \bmo_{d,d}(X_1\times X_2)$, the linear
    functional given by
    \begin{eqnarray}\label{dual pair}
        \ell(f)
        =\langle \varphi,f\rangle
        := \int \varphi(x_1,x_2)f(x_1,x_2) \, d\mu_1(x_1)d\mu_2(x_2),
    \end{eqnarray}
    initially defined on $H^1_{d,d}(X_1\times X_2)\cap
    L^2(X_1\times X_2)$, has a unique bounded extension to
    $H^1_{d,d}(X_1\times X_2)$ with
    \[
        \|\ell\|
        \leq C \mathcal{C}_1^d(\varphi).
    \]

    Conversely, every bounded linear functional $\ell$ on
    $H^1_{d,d}(X_1\times X_2)\cap L^2(X_1\times X_2)$ can be
    realized in the form of~\eqref{dual pair}, i.e., there exists $\varphi\in
    \bmo_{d,d}(X_1\times X_2)$ with such that \eqref{dual pair}
    holds for all $f\in H^1_{d,d}(X_1\times X_2)\cap L^2(X_1\times
    X_2)$, and
    \[
        \mathcal{C}_1^d(\varphi)
        \leq C\|\ell\|.
    \]
\end{thm}

\begin{proof}
Suppose $\varphi\in \bmo_{d,d}(X_1\times X_2)$ and we define
the linear functional $\ell$ as in (\ref{dual pair}) for $f\in
H^1_{d,d}(X_1\times X_2)\cap L^2(X_1\times X_2)$. Then by the
Haar expansion we have
\begin{eqnarray*}
    \ell(f)
    &=& \int \varphi(x_1,x_2)\sum_{Q_1\in\mathcal{D}_1} \sum_{Q_2\in\mathcal{D}_2}
        \sum_{u_1=1}^{M_{Q_1}-1}\sum_{u_2=1}^{M_{Q_2}-1}
        \langle f, h_{Q_1}^{u_1}h_{Q_2}^{u_2} \rangle
        h_{Q_1}^{u_1}(x_1)h_{Q_2}^{u_2}(x_2) \, d\mu_1(x_1)d\mu_2(x_2)\\
    & = & \sum_{Q_1\in\mathcal{D}_1} \sum_{Q_2\in\mathcal{D}_2}
        \sum_{u_1=1}^{M_{Q_1}-1}\sum_{u_2=1}^{M_{Q_2}-1}
        \langle f, h_{Q_1}^{u_1}h_{Q_2}^{u_2} \rangle
        \langle \varphi, h_{Q_1}^{u_1}h_{Q_2}^{u_2} \rangle\\
    & \leq & C\Big\|\big\{\langle f, h_{Q_1}^{u_1}h_{Q_2}^{u_2}\rangle\big\}\Big\|_{s^1}
        \Big\|\big\{\langle \varphi, h_{Q_1}^{u_1}h_{Q_2}^{u_2}\rangle\big\}\Big\|_{c^1},
\end{eqnarray*}
where the inequality follows from (\ref{L(s) bounded by s norm
and t norm}).

As a consequence, we obtain that
\begin{eqnarray*}
    |\ell(f)|
    & \leq & C\big\|T_L(f)\big\|_{s^1}\big\|T_L(\varphi)\big\|_{c^1}
    \leq C\|f\|_{H^1_{d,d}(X_1\times X_2)}\mathcal{C}_1^d(\varphi),
\end{eqnarray*}
which implies that $\varphi$ is a bounded linear functional on
$ H^1_{d,d}(X_1\times X_2)\cap L^2(X_1\times X_2)$ and hence
has a unique bounded extension to $H^1_{d,d}(X_1\times X_2)$
with $ \|\ell\|\leq C \mathcal{C}_1^d(\varphi) $ since $
H^1_{d,d}(X_1\times X_2)\cap L^2(X_1\times X_2)$ is dense in $
H^1_{d,d}(X_1\times X_2)$.

Conversely, suppose $\ell$ is a  bounded linear functional  on
$H^1_{d,d}(X_1\times X_2)\cap L^2(X_1\times X_2)$. We set
$\ell_1= \ell \circ T_P$. Then it is obvious that $\ell_1$ is a
bounded linear functional on $s^1\cap s^2$.

Note that $f=T_P\circ T_L(f)$ for every $f\in L^2(X_1\times
X_2)$. We have
\begin{eqnarray*}
    \ell(f)
    = \ell\big (T_P\circ T_L(f) \big)= \ell_1 \big(T_L(f)\big), \ \ {\rm where\ } \ell_1=\ell\circ T_P.
\end{eqnarray*}
Now by the duality of $s^1$ with $c^1$, we obtain that there
exists $t\in c^1$ with $\|t\|_{c^1}\leq \|\ell\|$ such that
\[
    \ell_1(s)
    =\langle t, s \rangle
\]
for all $s\in s^1$. Hence, we have
\begin{eqnarray*}
    \ell_1 \big(T_L(f)\big)&=& \langle t,T_L(f) \rangle  \\
    &=& \sum_{Q_1\in\mathcal{D}_1} \sum_{Q_2\in\mathcal{D}_2}
        \sum_{u_1=1}^{M_{Q_1}-1}\sum_{u_2=1}^{M_{Q_2}-1} t_{Q_1,u_1,Q_2,u_2}
        \langle f, h_{Q_1}^{u_1}h_{Q_2}^{u_2} \rangle \\
    &=& \bigg\langle f, \sum_{Q_1\in\mathcal{D}_1} \sum_{Q_2\in\mathcal{D}_2}
        \sum_{u_1=1}^{M_{Q_1}-1}\sum_{u_2=1}^{M_{Q_2}-1}
        t_{Q_1,u_1,Q_2,u_2} , h_{Q_1}^{u_1}h_{Q_2}^{u_2} \bigg\rangle \\
    &=& \big\langle f, T_P(t) \big\rangle,
\end{eqnarray*}
which implies that
\begin{eqnarray*}
    \ell(f)
    = \big\langle f, T_P(t) \big\rangle
\end{eqnarray*}
and moreover,
\[
    \mathcal{C}_1^d(T_P(t))
    \leq C\|t\|_{c^1}
    \leq C\|\ell\|.
    \qedhere
\]
\end{proof}

\section{Dyadic structure theorems for $H^1(\wX)$ and $\bmo(\wX)$}
\label{sec:multiparameter}
\setcounter{equation}{0}

In this final section, we prove our dyadic structure results.
Namely, the product Hardy space $H^1(\wX)$ is a sum of finitely
many dyadic product Hardy spaces, and the product $\bmo(\wX)$
space is an intersection of finitely many dyadic product $\bmo$
spaces. We also establish the atomic decomposition of the
dyadic product Hardy
spaces~$H^1_{\mathcal{D}_1^{t_1},\mathcal{D}_2^{t_2}}(\widetilde{X})$,
which plays a key role in the proofs of the dyadic structure
results. As in the previous section, these results all hold for
$n$~parameters, but for notational simplicity we write in terms
of two parameters.

\begin{thm}\label{thm structure of Hardy space}
    Let $\wX = X_1 \times X_2$ be a product space of homogeneous
    type. Then
    \[
        H^1(X_1\times X_2) =
        \sum_{t_1 = 1}^{T_1} \sum_{t_2 = 1}^{T_2}
        H^1_{\mathcal{D}_1^{t_1},\mathcal{D}_2^{t_2}}(X_1\times
        X_2).
    \]
    The corresponding result holds for $n$ parameters.
\end{thm}

As a consequence of Theorem \ref{thm structure of Hardy space},
we obtain the following by duality.
\begin{thm}\label{thm equivalence of BMO}
    Let $\wX = X_1\times X_2$ be a product space of homogeneous type.
    Then
    \[
        \bmo(\wX)
        = \bigcap_{t_1 = 1}^{T_1} \bigcap_{t_2 = 1}^{T_2}
            \bmo_{\mathcal{D}_1^{t_1},\mathcal{D}_2^{t_2}}(\wX).
    \]
        The corresponding result holds for $n$ parameters.

\end{thm}

To do this, we first give the definition of atoms for the
dyadic product Hardy space
$H^1_{\mathcal{D}_1^{t_1},\mathcal{D}_2^{t_2}}(\widetilde{X})$
and then provide the atomic decomposition, where
$t_1=1,\ldots,T_1;\ t_2=1,\ldots,T_2$. For brevity, we will
drop the reference to the parameters $(t_1,t_2)$ until we
return to the proof of Theorem~\ref{thm structure of Hardy
space}.

\begin{defn}\label{def of dyadic atom}
    A function $a(x_1,x_2)$ defined on $\widetilde{X}$ is called a
    \emph{dyadic atom of
    $H^1_{\mathcal{D}_1,\mathcal{D}_2}(\widetilde{X})$}
    if $a(x_1,x_2)$ satisfies:
    \begin{itemize}
        \item[(1)] $\supp a\subset\Omega$, where $\Omega$
            is an open set of $\widetilde{X}$ with finite
            measure;

        \item[(2)] $\|a\|_{L^2}\leq \mu(\Omega)^{-1/2}$;

        \item[(3)] $a$ can be further decomposed into
            rectangle atoms $a_R$ associated to dyadic
            rectangle $R=Q_1\times Q_2\in
            \mathcal{D}_1\times\mathcal{D}_2$, satisfying
            the following three conditions:

            \smallskip
            (i) $\supp a_R\subset R$ \quad
            (\emph{localization});

            \smallskip
            (ii) $\int_{ X_1 }a_R(x_1,x_2) \, d\mu_1(x_1) =
            0$ for a.e.~$x_2\in  X_2 $ and $\int_{ X_2
            }a_R(x_1,x_2) \, d\mu_2(x_2) = 0$

            \hskip.75cm for a.e.~$x_1\in  X_1 $ \quad
            (\emph{cancellation});

            \smallskip
            (iii) $a=\sum\limits_{R\in m(\Omega)}a_R$ and $
            \sum\limits_{R\in
            m(\Omega)}\|a_R\|_{L^2(\wX)}^2 \leq
            \mu(\Omega)^{-1}$ \quad (\emph{size}).
    \end{itemize}
\end{defn}

\begin{remark}
    We note that the only difference between the dyadic atoms
    defined here and the ``continuous'' atoms defined earlier
    in Definition~\ref{def-of-H1-atom} is that here the
    rectangle atoms $a_R$ are supported on the rectangles~$R$,
    while for the continuous atoms, the rectangle atoms $a_R$
    are supported on a dilate of~$R$.
\end{remark}

Next we provide the atomic decomposition for
$H^1_{\mathcal{D}_1,\mathcal{D}_2}(\widetilde{X})$.

\begin{thm}\label{thm atom for dyadic H1}
    If $f\in
    H^1_{\mathcal{D}_1,\mathcal{D}_2}(\widetilde{X})
    \cap L^2(\widetilde{X}) $, then
    \[
        f
        = \sum_k\lambda_ka_k,
    \]
    where each $a_k$ is an atom of
    $H^1_{\mathcal{D}_1,\mathcal{D}_2}(\widetilde{X}) $
    as in Definition \ref{def of dyadic atom}, and
    $\sum_k|\lambda_k|\leq
    C\|f\|_{H^1_{\mathcal{D}_1,\mathcal{D}_2}(\widetilde{X})
    }$.

    Conversely, suppose $f:=\sum_k\lambda_ka_k$ where each $a_k$ is
    an atom of
    $H^1_{\mathcal{D}_1,\mathcal{D}_2}(\widetilde{X}) $
    as in Definition \ref{def of dyadic atom}, and
    $\sum_k|\lambda_k|\leq C < \infty $, then $f\in
    H^1_{\mathcal{D}_1,\mathcal{D}_2}(\widetilde{X}) $
    and $
    \|f\|_{H^1_{\mathcal{D}_1,\mathcal{D}_2}(\widetilde{X})
    } \leq C \sum_k|\lambda_k|$.
\end{thm}

\begin{proof}
Suppose $f\in H^1_{\mathcal{D}_1,\mathcal{D}_2}(\widetilde{X})
$. Then $S_{d,d}(f)\in L^1(\widetilde{X})$. For each
$k\in\mathbb{Z}$, we now define
\begin{eqnarray*}
    \Omega_k
    &=& \{ (x_1,x_2)\in \widetilde{X}:
        S_{d,d}(f)(x_1,x_2)>2^k \},\\
    B_k
    &=& \big\{ R=Q_1\times Q_2\in\mathcal{D}_1\times \mathcal{D}_2: \mu(\Omega_k\cap R)>{1\over2}\mu(R),
        \ {\rm and}\ \mu(\Omega_{k+1}\cap R) \leq {1\over2}\mu(R)\big\}\\
    \widetilde{\Omega}_k
    &=& \{ (x_1,x_2)\in \widetilde{X}:
        M_s(\chi_{\Omega_k})(x_1,x_2)
        > \widetilde{C} \},
\end{eqnarray*}
where $M_s$ is the strong maximal function on $\widetilde{X}$
and $\widetilde{C}$ is a constant to be determined later.

Now by the Haar expansion convergent in $L^2(\wX)$, and since
each rectangle $R\in \mathcal{D}_1\times \mathcal{D}_2 $
belongs to exactly one set $B_k$, we have
\begin{eqnarray*}
    f
    &=& \sum_{Q_1\times Q_2 \in \mathcal{D}_1\times \mathcal{D}_2}
        \sum_{u_1=1}^{M_{Q_1}}\sum_{u_2=1}^{M_{Q_2}}
        \langle f,\ h_{Q_1}^{u_1}h_{Q_2}^{u_2} \rangle
        h_{Q_1}^{u_1}h_{Q_2}^{u_2}\\
   &=& \sum_k\sum_{Q_1\times Q_2\in B_k}\ \
        \sum_{u_1=1}^{M_{Q_1}}\sum_{u_2=1}^{M_{Q_2}}
        \langle f,\ h_{Q_1}^{u_1}h_{Q_2}^{u_2} \rangle
        h_{Q_1}^{u_1}h_{Q_2}^{u_2}\\
    &=:& \sum_k\lambda_ka_k,
\end{eqnarray*}
where
\[
    a_k(x_1,x_2)
    = {1\over \lambda_k}
        \sum_{R=Q_1\times Q_2 \in B_k}\ \
        \sum_{u_1=1}^{M_{Q_1}}\sum_{u_2=1}^{M_{Q_2}}
        \langle f,\ h_{Q_1}^{u_1}h_{Q_2}^{u_2} \rangle
        h_{Q_1}^{u_1}(x_1)h_{Q_2}^{u_2}(x_2)
\]
and
\begin{eqnarray*}\label{atom lambda k}
    \lambda_k
    = \Big(\sum_{R=Q_1\times Q_2 \in B_k}\ \
        \sum_{u_1=1}^{M_{Q_1}}\sum_{u_2=1}^{M_{Q_2}}
        \big|\langle f,\ h_{Q_1}^{u_1}h_{Q_2}^{u_2} \rangle\big|^2
        \Big)^{1/2}
        \mu(\widetilde{\Omega}_k)^{1/2}\hskip.7cm
\end{eqnarray*}

To see that the atomic decomposition $\sum_{k=-\infty}^\infty
\lambda_k a_k$ converges to $f$ in the $L^2$ norm, we only need
to show that $\|\sum_{|k|>\ell}\lambda_k a_k\|_2\rightarrow 0$
as $\ell\rightarrow \infty.$ This follows from the following
duality argument. Let $g\in L^2$ with $\|g\|_2 = 1$. Then
\[
    \big\|\sum_{|k|>\ell}\lambda_k a_k\big\|_2
    = \sup_{\|g\|_2=1}
        \big|\langle\sum_{|k|>\ell}\lambda_k a_k, g\rangle\big|.
\]
Note that
\begin{eqnarray*}
    \big\langle\sum_{|k|>\ell} \lambda_k a_k,g\big\rangle
    &=& \sum_{|k|>\ell}\sum_{R=Q_1\times Q_2 \in B_k}\ \
        \sum_{u_1=1}^{M_{Q_1}}\sum_{u_2=1}^{M_{Q_2}}
        \langle f,\ h_{Q_1}^{u_1}h_{Q_2}^{u_2} \rangle
        \langle g, h_{Q_1}^{u_1}h_{Q_2}^{u_2} \rangle.
\end{eqnarray*}
Applying H\"older's inequality gives
\begin{eqnarray*}
    \big|\big\langle\sum_{|k|>\ell}\lambda_k a_k, g\big\rangle\big|
    &\leq & \Big(\sum_{|k|>\ell}\sum_{R=Q_1\times Q_2 \in B_k}\ \
        \sum_{u_1=1}^{M_{Q_1}}\sum_{u_2=1}^{M_{Q_2}}
        \big|\langle f,\ h_{Q_1}^{u_1}h_{Q_2}^{u_2} \rangle\big|^2\Big)^{1/2}\\
    &&\times \Big(\sum_{|k|>\ell}\sum_{R=Q_1\times Q_2 \in B_k}\ \
        \sum_{u_1=1}^{M_{Q_1}}\sum_{u_2=1}^{M_{Q_2}}
        \big|\langle g,\ h_{Q_1}^{u_1}h_{Q_2}^{u_2} \rangle\big|^2\Big)^{1/2}
\end{eqnarray*}
Note again that
\[
    \Big(\sum_{|k|>\ell}\sum_{R=Q_1\times Q_2 \in B_k}\ \
        \sum_{u_1=1}^{M_{Q_1}}\sum_{u_2=1}^{M_{Q_2}}
        \big|\langle g,\ h_{Q_1}^{u_1}h_{Q_2}^{u_2} \rangle\big|^2\Big)^{1/2}
    \leq C\|g\|_2
    \leq C,
\]
and
\begin{eqnarray*}
    \Big(\sum_{|k|>\ell} \sum_{R=Q_1\times Q_2 \in B_k}\ \
    \sum_{u_1=1}^{M_{Q_1}} \sum_{u_2=1}^{M_{Q_2}}
    \big|\langle f,\ h_{Q_1}^{u_1}h_{Q_2}^{u_2} \rangle\big|^2\Big)^{1/2}
\end{eqnarray*}
tends to zero as $\ell$ tends to infinity. This implies that
$\|\sum_{|k|>\ell}\lambda_k a_k\|_2\rightarrow 0$ as
$\ell\rightarrow \infty$ and hence, the series
$\sum_{k=-\infty}^\infty \lambda_k a_k$ converges to $f$ in the
$L^2$ norm.

Next, it is easy to see that for each $k$, $\supp a_k\subset
\widetilde{\Omega}_k$ since $R\in B_k$ implies that $R\subset
\widetilde{\Omega}_k$ when we choose $\widetilde{C}<1/2$. Hence
we see that condition~(1) holds.

Now for each $k$, from the definition of $\lambda_k$ and the
H\"older's inequality, we have
\begin{eqnarray*}
    \|a_k\|_{L^2(\wX)}
    &=& \sup_{g:\ \|g\|_{L^2(\wX)}=1}
        \big|\langle a_k, g\rangle\big|
        = {1\over\lambda_k}\Big|\sum_{R=Q_1\times Q_2 \in B_k}\ \
        \sum_{u_1=1}^{M_{Q_1}}\sum_{u_2=1}^{M_{Q_2}}
        \langle f,\ h_{Q_1}^{u_1}h_{Q_2}^{u_2} \rangle
        \langle g, h_{Q_1}^{u_1}h_{Q_2}^{u_2} \rangle\Big|\\
    &\leq & {1\over\lambda_k}\Big(\sum_{R=Q_1\times Q_2 \in B_k}\ \
        \sum_{u_1=1}^{M_{Q_1}}\sum_{u_2=1}^{M_{Q_2}} \big|\langle f,\
        h_{Q_1}^{u_1}h_{Q_2}^{u_2} \rangle\big|^2\Big)^{1/2}\\
    &&\times \Big(\sum_{R=Q_1\times Q_2 \in B_k}\ \
        \sum_{u_1=1}^{M_{Q_1}}\sum_{u_2=1}^{M_{Q_2}}
        \big|\langle g,\ h_{Q_1}^{u_1}h_{Q_2}^{u_2} \rangle\big|^2\Big)^{1/2}\\
    &\leq& \mu(\widetilde{\Omega}_k)^{-1/2},
\end{eqnarray*}
which implies that condition~(2) holds.

It remains to check that $a_k$ satisfies condition~(3) of
Definition~\ref{def of dyadic atom}. To see this, we can
further decompose $a_k$ as
$$
   a_k
   = \sum_{\overline{R}\in m(\widetilde{\Omega}_k)} a_{k, \overline{R}},
$$
where $m(\widetilde{\Omega}_k)$ denotes the collection of
maximal dyadic rectangles $\overline{R}\in \mathcal{D}_1\times
\mathcal{D}_2$ contained in $\widetilde{\Omega}_k$, and
\begin{eqnarray*}
    a_{k, \overline{R}}(x_1,x_2)
    &=& {1\over \lambda_k }\ \
        \sum_{R=Q_1\times Q_2 \in B_k,\ R\subset \overline{R}}\ \
        \sum_{u_1=1}^{M_{Q_1}}\sum_{u_2=1}^{M_{Q_2}} \langle f,\
        h_{Q_1}^{u_1}h_{Q_2}^{u_2} \rangle
        h_{Q_1}^{u_1}(x_1)h_{Q_2}^{u_2}(x_2) .
\end{eqnarray*}

By definition we can verify that
$$
    {\rm supp}\,a_{k,\overline{R}}
    \subset \overline{R}.
$$
By the facts that $\int h_{Q_1}^{u_1}(x_1) d\mu_1(x_1)=\int
h_{Q_2}^{u_2}(x_2)d\mu_2(x_2)=0,$ we have, for a.e. $x_2\in
X_2 $,
$$
    \int_{ X_1 } a_{k,\overline{R}}(x_1,x_2)d\mu_1(x_1)
    = 0
$$
and for a.e.~$x_1\in  X_1 $,
$$
    \int_{ X_2 } a_{k,\overline{R}}(x_1,x_2)d\mu_2(x_2)
    = 0,
$$
which yield that the conditions (i) and (ii) of (3) in
Definition~\ref{def of dyadic atom} hold. It remains to show
that $a_k$ satisfies the condition~(iii) of~(3).

To see this, we first note that
\begin{eqnarray*}
    \|a_{k,\overline{R}}\|_{2}
    &=& {1\over\lambda_k} \Big\|
        \sum_{R=Q_1\times Q_2 \in B_k,\ R\subset \overline{R}}\ \
        \sum_{u_1=1}^{M_{Q_1}}\sum_{u_2=1}^{M_{Q_2}} \langle f,\
        h_{Q_1}^{u_1}h_{Q_2}^{u_2} \rangle
        h_{Q_1}^{u_1}(x_1)h_{Q_2}^{u_2}(x_2)\Big\|_2.
\end{eqnarray*}
From the definition of $\lambda_k$, by applying the same
argument as for the estimates of $\|a_k\|_{2}$, we can obtain
that
$$
    \sum_{\overline{R}\in m(\widetilde{\Omega}_k)}
        \big\|a_{k,\overline{R}}\big\|^2_{L^2}
    \leq \mu(\widetilde{\Omega}_k)^{-1/2},
$$
which implies that condition~(iii) holds.

We now prove that $\sum_k|\lambda_k| \leq
C\|f\|_{H^1_{\mathcal{D}_1,\mathcal{D}_2}(\widetilde{X}) }$.

First note that by definition of $\Omega_{k+1}$ and
Remark~\ref{remark:duality seq spaces},
$$
    \int_{\widetilde{\Omega}_k\setminus \Omega_{k+1}}
        S_{d,d}(f)(x_1,x_2)^2\,d\mu_1(x_1)d\mu_2(x_2)
    \leq  2^{2(k+1)}\mu(\widetilde{\Omega}_k)   \leq  C2^{2(k+1)}\mu(\Omega_k).
$$
Moreover, by the Monotone Convergence Theorem and by
Remark~\ref{remark:duality seq spaces}, we have
\begin{eqnarray*}
    \lefteqn{\int_{\widetilde{\Omega}_k\setminus \Omega_{k+1}}
        S_{d,d}(f)(x_1,x_2)^2\,d\mu_1(x_1)d\mu_2(x_2)} \hspace{1cm}\\
    &&\geq \int_{\widetilde{\Omega}_k\setminus \Omega_{k+1}}
        \sum_{R=Q_1\times Q_2\in B_k}\sum_{u_1=1}^{M_{Q_1}}\sum_{u_2=1}^{M_{Q_2}}
        \big| \langle f, h_{Q_1}^{u_1}h_{Q_2}^{u_2}\rangle \big|^2
        \widetilde{\chi}_{Q_1}(x_1)\widetilde{\chi}_{Q_2}(x_2)\,d\mu_1(x_1)d\mu_2(x_2)\\
    &&\geq \sum_{R=Q_1\times Q_2\in B_k}\sum_{u_1=1}^{M_{Q_1}}\sum_{u_2=1}^{M_{Q_2}}
        \big| \langle f, h_{Q_1}^{u_1}h_{Q_2}^{u_2}\rangle \big|^2
        {\mu(\widetilde{\Omega}_k\setminus \Omega_{k+1}\, \cap R )\over \mu(R)}\\
    &&\geq 1/2 \sum_{R=Q_1\times Q_2\in B_k}\sum_{u_1=1}^{M_{Q_1}}\sum_{u_2=1}^{M_{Q_2}}
        \big| \langle f, h_{Q_1}^{u_1}h_{Q_2}^{u_2}\rangle \big|^2,
\end{eqnarray*}
since $R\in B_k$ implies that $\mu(
(\widetilde{\Omega}_{k+1}\backslash \Omega_{k+1}) \, \cap R
)\leq \mu(R)/2$.

From the definition of $\lambda_k$, we have
\begin{eqnarray*}
    \sum_k|\lambda_k|
    &=& C\sum_k\Big(\sum_{R=Q_1\times Q_2 \in B_k}\ \
        \sum_{u_1=1}^{M_{Q_1}} \sum_{u_2=1}^{M_{Q_2}}
        \big|\langle f,\ h_{Q_1}^{u_1}h_{Q_2}^{u_2} \rangle\big|^2\Big)^{1/2}
        \mu(\widetilde{\Omega}_k)^{1/2}\\
    &\leq& C\sum_k 2^{k+1} \mu(\widetilde{\Omega}_k)^{1/2}
        \mu(\widetilde{\Omega}_k)^{1/2}\\
    &\leq& C\sum_k2^{k}\mu(\Omega_k)\\
    &\leq& C\|f\|_{H^1_{\mathcal{D}_1,\mathcal{D}_2}(\widetilde{X})}.
\end{eqnarray*}

Conversely, suppose $f := \sum_k \lambda_k a_k$ where each
$a_k$ is an atom of
$H^1_{\mathcal{D}_1,\mathcal{D}_2}(\widetilde{X}) $ as in
Definition~\ref{def of dyadic atom}, and $\sum_k|\lambda_k|
\leq C < \infty $.

Then
\begin{eqnarray*}
    S_{d,d}(f)(x_1,x_2)
    = S_{d,d}\big(\sum_k\lambda_k a_k\big)(x_1,x_2)
    \leq \sum_k|\lambda_k| S_{d,d}\big(a_k\big)(x_1,x_2).
\end{eqnarray*}
It suffices to estimate the $L^1(\widetilde{X}) $ norm of
$S_{d,d}\big(a_k\big)(x_1,x_2)$. For the sake of simplicity, we
drop the subscript $k$ for the atom and now estimate
$\|S_{d,d}(a)\|_{L^1(\widetilde{X})}$, where $a$ is a dyadic
atom as in Definition~\ref{def of dyadic atom} with
support~$\Omega$.

Now let $ \bar{\Omega} = \cup_{R\in m(\Omega)} 100R $.  We have
\begin{eqnarray*}
    \|S_{d,d}\big(a\big)\|_{L^1(\widetilde{X})}
    = \|S_{d,d}\big(a\big)\|_{L^1(\bar{\Omega})}
        + \|S_{d,d}\big(a\big)\|_{L^1(\bar{\Omega}^c)}.
\end{eqnarray*}

It is obvious that $\|S_{d,d}\big(a\big)\|_{L^1(\bar{\Omega})}
\leq \mu(\bar{\Omega})^{1/2}
\|S_{d,d}\big(a\big)\|_{L^2(\bar{\Omega})} \leq C
\mu(\bar{\Omega})^{1/2} \|a\|_{L^2(\widetilde{X})} \leq C$,
where $C$ is a positive constant independent of~$a$.

From the definition of the atom~$a$, we have
\begin{eqnarray*}
    \|S_{d,d}\big(a\big)\|_{L^1(\bar{\Omega}^c)}
    &\leq& \sum_{R\in m(\Omega)}\int_{(100R)^c}
         S_{d,d}(a_R)(x_1,x_2) \, d\mu_1(x_1) \, d\mu_1(dx_2).
\end{eqnarray*}
For each rectangle~$R = Q_1 \times Q_2$ that appears in this
sum, we have from the definition of the square function
$S_{d,d}(a_R)$ that
\begin{eqnarray*}
    S_{d,d}(a_R)(x_1,x_2)
    &:=& \Big\{\sum_{Q'_1\in\mathcal{D}_1} \sum_{Q'_2\in\mathcal{D}_2}
        \sum_{u_1=1}^{M_{Q'_1}-1}\sum_{u_2=1}^{M_{Q'_2}-1}
        \big| \langle a_R, h_{Q'_1}^{u_1}h_{Q'_2}^{u_2}\rangle
        \widetilde{\chi}_{Q'_1}(x_1)\widetilde{\chi}_{Q'_2}(x_2)\big|^2
        \Big\}^{1/2}.
\end{eqnarray*}
Now if $Q_2$ is a child of~$Q'_2$, then $h_{Q'_2}$ is constant
on $Q_2$, and so
\begin{eqnarray*}
    \langle a_R, h_{Q'_1}^{u_1} h_{Q'_2}^{u_2} \rangle
    &=& \int_{Q_1} \! \int_{Q_2} a_R(x_1,x_2) h_{Q'_1}^{u_1}(x_1) h_{Q'_2}^{u_2}(x_2) \, d\mu_1(x_1) \, d\mu_2(x_2)\\
    &=& \int_{Q_1} C \left[\int_{Q_2} a_R(x_1,x_2) \, d\mu_2(x_2)\right] h_{Q'_1}^{u_1}(x_1) \, d\mu_1(x_1)\\
    &=& 0
\end{eqnarray*}
by the cancellation property of~$a_R$. Similarly, if $Q_1$ is a
child of~$Q'_1$, then $\langle a_R, h_{Q'_1}^{u_1}
h_{Q'_2}^{u_2} \rangle = 0$. Clearly if $Q_i\cap Q'_i
=\emptyset$ for $i=1$ or $i=2$ then the inner product is also
zero.
Thus
\begin{eqnarray*}
    S_{d,d}(a_R)(x_1,x_2)
    &=& \Big\{\sum_{\substack{Q'_1\in\mathcal{D}_1,\\ Q'_1\subseteq Q_1}}
    \sum_{\substack{Q'_2\in\mathcal{D}_2,\\ Q'_2\subseteq Q_2}}
        \sum_{u_1=1}^{M_{Q'_1}-1}\sum_{u_2=1}^{M_{Q'_2}-1}
        \big| \langle a_R, h_{Q'_1}^{u_1}h_{Q'_2}^{u_2}\rangle
        \widetilde{\chi}_{Q'_1}(x_1)\widetilde{\chi}_{Q'_2}(x_2)\big|^2 \Big\}^{1/2} .
\end{eqnarray*}
This implies that $S_{d,d}(a_R)(x_1,x_2) $ is supported in~$R$.
Thus,
\[
    \int_{(100R)^c}
    S_{d,d}(a_R)(x_1,x_2) \, d\mu_1(x_1) \, d\mu_2(x_2)
    = 0.
\]

As a consequence, we have
\begin{eqnarray*}
    \|S_{d,d}\big(a\big)\|_{L^1(\widetilde{X})}
    \leq C,
\end{eqnarray*}
where $C$ is a positive constant independent of $a$. Then
\begin{equation*}
   \| S_{d,d}(f)\|_{L^1(\widetilde{X})}
   \leq \sum_k|\lambda_k| \|S_{d,d}\big(a_k\big)\|_{L^1(\widetilde{X})}
   \leq C\sum_k|\lambda_k|.
   \qedhere
\end{equation*}
\end{proof}

It remains to prove Theorem~\ref{thm structure of Hardy space}.

\begin{proof}[Proof of Theorem \ref{thm structure of Hardy space}]
First we prove the inclusion $H^1(X_1\times X_2) \subset
\sum_{t_1=1}^{T_1}\sum_{t_2=1}^{T_2}
H^1_{\mathcal{D}_1^{t_1},\mathcal{D}_2^{t_2}}(X_1\times X_2)$.
Suppose $f\in H^1(X_1\times X_2)\cap L^2(X_1\times X_2)$. From
Theorem \ref{theorem-Hp atom decomp}, we obtain that
\begin{eqnarray*}
    f
    = \sum_j\lambda_j a_j,
\end{eqnarray*}
where each $a_j$ is an $H^1(X_1\times X_2)$ atom as defined in
Definition~\ref{def-of-H1-atom}, and $\sum_j|\lambda_j| \leq
C\|f\|_{H^1(X_1\times X_2)}$.

Thus, from the Definition \ref{def-of-H1-atom}, we see that
$a_j$ is supported in an open set $\Omega_j\in X_1\times X_2$
with finite measure. Moreover, $a_j=\sum_{R\in m(\Omega_j)}
a_{j,R}$, where each $a_{j,R}$ is supported in $\overline{C}R$.
Since $R$ is a dyadic rectangle, we have
$R=Q_{\alpha_1}^{k_1}\times Q_{\alpha_2}^{k_2}$ for some
$k_1,k_2\in \mathbb{Z}$ and $\alpha_1\in \mathscr{A}_{k_1}$ and
$\alpha_2\in \mathscr{A}_{k_2}$. Thus, from the definition of
the dyadic cubes in Section 2.1, we further have $R\subset
B(x_{\alpha_1}^{k_1},C_1\delta^{k_1})\times
B(x_{\alpha_2}^{k_2},C_2\delta^{k_2})$. As a consequence, we
obtain that supp$a_{j,R}$ is contained in $
B(x_{\alpha_1}^{k_1},\overline{C}C_1\delta^{k_1})\times
B(x_{\alpha_2}^{k_2},\overline{C}C_2\delta^{k_2})
=\overline{C}R$.

Next, from Lemma~4.12 in \cite{HK}, we see that
$B(x_{\alpha_1}^{k_1},\overline{C}C_1\delta^{k_1})$ must be
contained in some $Q_1^{t_1}(\overline{k}_1,\beta_1) \in
\mathcal{D}^{t_1}$ for some $t_1$ in $\{1, 2, \ldots, T_1\}$.
Similarly, $B(x_{\alpha_2}^{k_2},\overline{C}C_2\delta^{k_2})$
must be contained in $Q_2^{t_2}(\overline{k}_2,\beta_2) \in
\mathcal{D}^{t_2}$ for some $t_2$ in $\{1, 2, \ldots, T_2\}$.
Moreover, there exists a constant $\widetilde{C}$ such that
$\mu(Q_1^{t_1}(\overline{k}_1,\beta_1) ) \leq
\widetilde{C}\mu(B(x_{\alpha_1}^{k_1},\overline{C}C_1\delta^{k_1}))
$ and that $\mu(Q_2^{t_2}(\overline{k}_2,\beta_2) ) \leq
\widetilde{C}\mu(B(x_{\alpha_2}^{k_2},\overline{C}C_2\delta^{k_2}))
$. As a consequence, $\supp a_{j,R}$ is contained in
$Q_1^{t_1}(\overline{k}_1,\beta_1) \times
Q_2^{t_2}(\overline{k}_2,\beta_2)$.

Thus, we can divide the summation $a_j=\sum_{R\in m(\Omega_j)}
a_{j,R}$ as follows:
\begin{eqnarray*}
    a_j
    = \sum_{t_1=1}^{T_1}\sum_{t_2=1}^{T_2}\sum_{R\in m(\Omega_j),
        \overline{C}R \subset Q_1^{t_1}(\overline{k}_1,\beta_1)
        \times Q_2^{t_2}(\overline{k}_2,\beta_2)}
        a_{j,R}.
\end{eqnarray*}
We now set
$$
   a_{j,t_1,t_2}
   = \sum_{R\in m(\Omega_j),
   \overline{C}R \subset Q_1^{t_1}(\overline{k}_1,\beta_1)
   \times Q_2^{t_2}(\overline{k}_2,\beta_2)}
   a_{j,R}.
$$
Then we can verify that $a_{j,t_1,t_2}$ is an
$H^1_{\mathcal{D}^{t_1}, \mathcal{D}^{t_2}}(X_1\times X_2)$
atom. Moreover, $f_{t_1,t_2}:=\sum_j\lambda_{j}
a_{j,t_1,t_2}\in H^1_{\mathcal{D}^{t_1},
\mathcal{D}^{t_2}}(X_1\times X_2)$ with
$$
    \|f_{t_1,t_2}\|_{H^1_{\mathcal{D}^{t_1}, \mathcal{D}^{t_2}}(X_1\times X_2)}
    \leq \sum_j|\lambda_j|
    \leq C\|f\|_{H^1(X_1\times X_2)}.
$$

Hence, we obtain that
\begin{eqnarray*}
    f
    = \sum_j\lambda_j a_j
    = \sum_{t_1=1}^{T_1}\sum_{t_2=1}^{T_2} f_{t_1,t_2},
\end{eqnarray*}
where $f_{t_1,t_2}\in H^1_{\mathcal{D}^{t_1},
\mathcal{D}^{t_2}}(X_1\times X_2)$, and $\sum_{t_1=1}^{T_1}
\sum_{t_2=1}^{T_2} \|f_{t_1,t_2}\|_{H^1_{\mathcal{D}^{t_1},
\mathcal{D}^{t_2}}(X_1\times X_2)} \leq
CT_1T_2\|f\|_{H^1(X_1\times X_2)}.  $

This implies that $H^1(X_1\times X_2)\cap L^2(X_1\times
X_2)\subset \sum_{t_1=1}^{T_1}\sum_{t_2=1}^{T_2}
H^1_{\mathcal{D}_1^{t_1},\mathcal{D}_2^{t_2}}(X_1\times X_2)$.
Moreover, note that $H^1(X_1\times X_2)\cap L^2(X_1\times X_2)$
is dense in $H^1(X_1\times X_2)$, we obtain that $H^1(X_1\times
X_2)\subset \sum_{t_1=1}^{T_1}\sum_{t_2=1}^{T_2}
H^1_{\mathcal{D}_1^{t_1},\mathcal{D}_2^{t_2}}(X_1\times X_2)$.

Second, we prove that
$H^1_{\mathcal{D}_1^{t_1},\mathcal{D}_2^{t_2}}(\widetilde{X})\subset
H^1(\widetilde{X})$, for every $t_1=1,\ldots,T_1;\
t_2=1,\ldots,T_2$. To see this, we now apply the atomic
decomposition for the dyadic product Hardy space
$H^1_{\mathcal{D}_1^{t_1},\mathcal{D}_2^{t_2}}(\widetilde{X})$,
i.e.,  Theorem \ref{thm atom for dyadic H1}.

For $f\in H^1_{\mathcal{D}_1^{t_1},
\mathcal{D}_2^{t_2}}(\widetilde{X}) \cap L^2(\widetilde{X})$,
we have $f=\sum_{k}\lambda_k a_k$, where each $a_k$ is a dyadic
atom as in Definition~\ref{def of dyadic atom} and
$\sum_k|\lambda_k|\leq C \|f\|_{H^1_{\mathcal{D}_1^{t_1},
\mathcal{D}_2^{t_2}}(\widetilde{X})}$.

Note that the dyadic atom in Definition \ref{def of dyadic
atom} is a special case of the $H^1(\widetilde{X})$ atom. We
obtain that $f \in H^1(\widetilde{X})$ and that
$\|f\|_{H^1(\widetilde{X})} \leq
C\|f\|_{H^1_{\mathcal{D}_1^{t_1},
\mathcal{D}_2^{t_2}}(\widetilde{X})}$, which yields that
$H^1_{\mathcal{D}_1^{t_1}, \mathcal{D}_2^{t_2}}(\widetilde{X})
\cap L^2(\widetilde{X}) \subset H^1(\widetilde{X})$.

As a consequence, we obtain that
$H^1_{\mathcal{D}_1^{t_1},\mathcal{D}_2^{t_2}}(\widetilde{X})
\subset H^1(\widetilde{X})$.
\end{proof}



\end{document}